\documentclass[10pt]{article}
\usepackage{amsmath,amsthm,amscd,amssymb,mathrsfs,setspace}
\usepackage{latexsym,epsf,epsfig}
\usepackage[hmargin=3cm,vmargin=3.5cm]{geometry}
\usepackage{color}

\newcommand{\ds}{\displaystyle}

\newcommand{\R}{{\mathbb R}}
\newcommand{\reals}{\mathbb{R}}
\newcommand{\realstwo}{\mathbb{R}^2}
\newcommand{\realsthree}{\mathbb{R}^3}
\newcommand{\xb}{{\bf{x}}}

\newcommand{\Dn}{\partial_{\nu}}

\newcommand{\cE}{{\mathscr{E}}}
\newcommand{\cF}{{\mathcal{F}}}

\newcommand{\bH}{\mathbf{H}}
\newcommand{\Om}{\Omega}
\newcommand{\om}{\omega}
\newcommand{\ga}{\gamma}
\newcommand{\s}{\sigma}
\newcommand{\e}{\epsilon}
\newcommand{\Ez}{E_{z}}
\newcommand{\bA}{\mathbf{A}}

\theoremstyle{plain}
\newtheorem{theorem}{Theorem}[section]
\newtheorem{lemma}[theorem]{Lemma}
\newtheorem{proposition}[theorem]{Proposition}
\newtheorem{corollary}[theorem]{Corollary}
\newtheorem{remark}[theorem]{Remark}
\newtheorem{assumption}[theorem]{Assumption}
\numberwithin{equation}{section}
\numberwithin{theorem}{section}
\title{Attractors for Delayed, Non-Rotational von Karman Plates \\ with
 Applications to Flow-Structure Interactions \\ Without any Damping}

\date{\today}

 \author{\begin{tabular}[t]{c@{\extracolsep{1em}}c@{\extracolsep{1em}}c}
      Igor Chueshov &      Irena Lasiecka & Justin T. Webster  \\ \it Kharkov National University &
 \it University of Virginia & \it Oregon State University \\ \it Kharkov, UA  &
 \it Charlottesville, VA &\it Corvallis, OR\\ \it chueshov@univer.kharkov.ua &
 \it il2v@virginia.edu & \it websterj@math.oregonstate.edu
\end{tabular}}

\begin{document}
\maketitle

\begin{abstract} {\noindent
 This paper is devoted to  a  long time behavior analysis associated with   flow structure interactions at subsonic and supersonic velocities. It turns out that  an intrinsic  component of that analysis  is  the study of attracting sets corresponding to von Karman plate equations with  {\it delayed terms} and {\it  without rotational}  terms. The presence of delay terms in the  dynamical system   leads to the loss of gradient structure while the absence of rotational terms in von Karman plates  leads to the loss of compactness of the orbits.
 Both  these features
 make the analysis of long time behavior rather subtle rendering  the  established tools in the theory of PDE dynamical systems not  applicable. 
  It is our goal  to develop methodology that is capable of handling  this class of problems.

\smallskip\par
\noindent {\bf Key terms:} nonlinear plate, PDE with delay,
 long-time behavior of solutions, dynamical systems, global attractors, flow-structure interaction,
 \smallskip\par
 }
\end{abstract}
\section{Introduction}
The study of  von Karman plates in the presence of
aerodynamical forces  represented
by some  {\em delayed} functional  is physically motivated \cite{bolotin,dowell,dowellnon}.These
  models with delay often arise in the modeling of coupled dynamics (e.g., fluid or flow-structure interaction) where  the impact of the off-plate dynamics can
  be   written as  a boundary value of some delayed (flow) potential in the plate equation after a sufficiently large time. In fact, this is the case for flow-plate interactions arising in the modeling of panels and plates immersed in an inviscid flow (for some discussion, see \cite{b-c-1} and  \cite{springer}), and thus
  we can reduce  the study of long-time behavior of solutions to the coupled flow-plate system  to the  problem
of a von Karman plate forced by some delay term.

In the present work, the motivation and significance of studying this class of models derives from
recent developments in the area of flow-structure interactions,  with the goal of attaining good  mathematical  understanding of flow-structure dynamics at both subsonic and supersonic flow velocities.
It is known from experiment (and also confirmed by numerics), that the
potential flow (particularly at the supersonic speeds) has the ability of inducing a certain amount of stability in the moving structure. This is the case \textit{even} when the structure itself does not possess mechanical  damping mechanisms.
 If one writes down the equations for the interactive system, along with  the  standard energy balance, this dissipative effect is not exhibited  at all;
quantities are conserved and not dissipated.  Thus, there must be some ``hidden" mechanism which produces this dissipation.  It is our goal to shed some light on this phenomenon.
As it turns out, the decoupling technique introduced in \cite{LBC96,b-c-1},  which reduces the analysis
of full flow-structure interaction to that of a certain delayed plate model, allows us to observe certain stabilizing effects of the flow. These occur in the form of non-conservative forces acting upon the structure  as  the ``downwash" of the flow.
This idea was already applied  to Berger plate models \cite{oldchueshov1,oldchueshov2}
 in the proof  of  existence  of  attractors corresponding to the associated reduced plate problem with a delayed term.

In fact, well-posedness and long-time behavior analyses of nonlinear plate PDEs with delays  have been treated in
\cite{Chu92b} (see also \cite{springer}): first, in the case of the von Karman model \textit{with} rotational inertia,
and secondly, in  \cite{oldchueshov1,oldchueshov2}, in the case of the Berger model with a small intensity of delayed term (this corresponds to a large speed $U$ of the flow of gas - hypersonic).
 These expositions flesh out the existence and properties of global attractors for the general plate with delay in the presence of a `natural' form of interior damping, and then apply this general result to the specific delayed
 (aeroelastic) force given in the full flow-plate coupling.

 It should be noted that the presence of rotational inertia parameter, while drastically improving the topological properties
 of the model, is neither natural nor desirable in the context of flow-structure interaction.
 First, the original model for flow structure interaction describes the interaction between the mid-surface of the  plate and flow above the plate; and hence should the plate should be treated two-dimensionally so the equation describing the ``downwash" of the flow and the oscillations of the plate can agree on the interface. Thus, the rotational inertia term (proportional to the cube of the plate thickness, see, e.g., \cite{lagnese}) should be neglected. Secondly, the presence of rotational term changes the kinetic energy in the system, forcing a much stronger stabilization mechanism (abstractly, this corresponds to the well-known fact that the essential spectrum of an operator can not be moved with a compact perturbation). In that case, the stability induced by the flow (a viscous, velocity-proportional damping) does not suffice providing a stabilizing effect on the structure. Instead, when the rotational inertia term is neglected, the damping secured by the \textit{flow alone} provides the main mechanism for stabilization. This would seem to corroborate physical findings.
 In view of this, it is paramount to the problem at hand to consider the model {\it  which does not account for rotational inertia} and imposes {\it no limiting regimes}  on the  flow velocity parameter $U \geq 0$, $U \ne 1$.

\par
It is the purpose of this paper to show that the   requirement of  the  rotational term and restrictions on the  values of $U$
 can be eliminated. This is accomplished  by resorting to  modern analytical tools, including compensated compactness methods. These tools, developed  over the last few  years \cite{springer},  are capable  of revealing a mathematical  structure of the problem which  is  consistent with numerical and experimental findings.
Thus, not only the results, but also the techniques and approach utilized in this paper appear to be significant to the discussion of flow-plate interactions and aeroelasticity.
Accordingly,  in this treatment,  we focus on the more difficult, and hereto open, \textit{non-rotational}
  von Karman case, {\it without} any restrictions on value of unperturbed speed ($0\le U<+\infty$, $U\neq 1$).

   The mathematical difficulties which arise in this model force us to consider new long-time behavior technologies applied within this framework. In order to provide a glimpse of this, and to demonstrate the  timeliness of the project undertaken, it suffices to note that  the  first (and most fundamental)  difficulty one is  faced with is the well-posedness of finite energy solutions corresponding to flow-structure interactions - in particularly in the supersonic case.  Results and methods found in past literature depend critically on the presence of {\it rotational inertial term}, see \cite{LBC96,b-c-1} and also \cite[Chapters 6 and 12]{springer}.
 Only recently has this problem has been solved for the non-rotational model for {\it all} flow velocities \cite{supersonic}
 (the subsonic case was also discussed earlier in \cite{b-c,webster} and also in \cite{springer}). Equipped with dynamical system structure for the flow-plate interaction, one may then proceed  with a study of long time behavior.
\par

The main  points  of this treatment are: (i) to give a unified approach to the long-time behavior of this model, beginning with well-posedness discussions, and working to the existence and finite-dimensionality of a global attractor. Our results are novel, and they complete the analysis of long-time behavior of von Karman systems with delay by addressing the non-rotational case and the associated resulting lack of compactness. (ii) We make use of a recent technique \cite{chlJDE04,ch-l,springer,cl-hcte,pelin} which allows us to obtain the asymptotic compactness property for the dynamical system without making use of any gradient type structure of the dynamics (not available in this model, owing to the dispersive flow term). This has the added benefit of producing extra regularity of the attractor, and is a less demanding approach (the `traditional' approach of showing quasistability  (on the attractor)
for von Karman plates requires a full characterization of the attractor \textit{or} assumptions on finiteness of the stationary points of the dynamical system \textit{and} the use of backward-in-time methods, which are not available in systems with delay). Lastly, (iii) we make use of new breakthroughs (both in well-posedness and hidden compactness) with regard to von Karman flow-plate interactions in order to apply our general results on this system with delay to an aeroelasticity model which has received considerable attention from the PDE and dynamical systems community, as well as the engineering and aeroscience community, over the past 20 years. The  results following from the application are also novel, and provide a complete treatment of   the long-time behavior of 
a von Karman  full flow-plate interaction model.
Namely, for this model we show  the existence of a finite-dimensional compact set in the phase space of
  the plate component. This set attracts all plate trajectories  in the case when the corresponding
  gas flow initial data are localized. We emphasize that this result
 {\it does not assume any damping imposed}  on the system.
  While  stability
 of flow structure interactions {\it without any damping}  has been experienced both numerically and experimentally, our works appears to be a  first
 rigorous mathematical treatment of this phenomena. The key ingredients of the analysis include:
 exhibit of "hidden" dissipation related to the dispersive character of the flow equation along with "hidden" regularity
 of the boundary  traces of  the flow, the decoupling technique introduced in \cite{b-c-1}  and  applied within the context
 of recent and powerful techniques aimed at showing asymptotic smoothness and quasi-stability \cite{springer} without
 the a-priori known  compactness or gradient structure.

\par

We also note that
stability and flutter control in flow structure interactions has been also  treated
by several  authors at different levels (see, e.g., \cite{bal0,bolotin,dowell1,HP02} and the references therein).
For instance the recent monograph \cite{bal0} contains  a detailed derivation   of the flow equation coupled to  a linear beam equation with the boundary data which involve the Kutta-Jukovsky circulation
condition for the flow and hence can be applied to elastic wings dynamics.
 The main trust of the corresponding analysis (see also \cite{shubov1} and the literature cited therein) is the characterization of  unstable aeroelastic modes arising in Possio equation  which is linear  integral equation describing aerodynamic pressure on the structure. The  associated spectral  analysis is focused on finding unstable modes in the linear dynamics. The obtained results \cite{bal0,shubov1}  provide  specific information on the flutter speed.
In contrast, our results show that  long time behavior of a nonlinear flow-plate  model can  be reduced to a finite dimensional attracting set.   More precise information on   the stability/unstability of {\it  finite dimensional} orbits would require
in depth study of the resulting {\it nonlinear finite dimensional  dynamical system}  (which can be chaotic in the supersonic case).

\medskip\par
  The paper is organized as follows.
  In Section \ref{Sect2} we describe the general
  flow-plate interaction models and its relation to the model with delay. We also state well-posedness results here (Proposition~\ref{p:well})
  and sketch a proof. This result allows us to  define a corresponding evolution semigroup.
  In Section~\ref{Sect3} we state and discuss our main result on long-time dynamics of
the delayed model (Theorem~\ref{maintheorem}). We also show what consequences
this result yields for long-time dynamics of the general
  flow-plate interaction model (Theorem~\ref{th:main2}).
The next section, Section~\ref{Sect4}, is central and devoted to the proofs of the main results.
In Section~\ref{flowplateint} we briefly discuss the proof of the reduction theorem (Theorem~\ref{rewrite}) which rigorously ties the full flow-plate dynamics to the evolution of the von Karman plate with delay.
Finally, in the Appendix
we establish some needed properties of the delayed (aerodynamic type) force.

\smallskip \par\noindent {\bf Notation:}
For the remainder of the text we write $\xb$ for $(x,y,z) \in \realsthree_+$ or $(x,y) \in \Omega \subset \realstwo_{\{(x,y)\}}$, as dictated by context. Norms $||\cdot||$ are taken to be $L_2(D)$ for the domain dictated by context. Inner products in $L_2(\realsthree_+)$ are written $(\cdot,\cdot)$, while inner products in $L_2(\reals^2\equiv \partial \realsthree_+)$ are written $<\cdot,\cdot>$. Also, $ H^s(D)$ will denote the Sobolev space of order $s$, defined on a domain $D$, and $H^s_0(D)$ denotes the closure of $C_0^{\infty}(D)$ in the $H^s(D)$ norm
which we denote by $\|\cdot\|_{H^s(D)}$ or $\|\cdot\|_{s,D}$. We make use of the standard notation for the trace of functions defined on $\realsthree_+$, i.e. for $\phi \in H^1(\realsthree_+)$, $\gamma[\phi]=\phi \big|_{z=0}$ is the trace of $\phi$ on the plane $\{\xb:z=0\}$.

\section{Motivation and Description of the Model}\label{Sect2}

\subsection{Flow-Plate Interactions}\label{subs-flowplateint}

The model we begin  with describes the interaction between a nonlinear plate with a field or flow of gas above it. To describe the behavior of the gas we make use of the theory of potential flows (see, e.g., \cite{bolotin,dowell,kras} and the references therein) which produces a perturbed wave equation for the velocity potential of the flow. The oscillatory behavior of the plate is governed by the second order (in time) Kirchhoff plate equation with a general nonlinearity.
  We consider the von Karman nonlinearity, which is used in the modeling of the large oscillations of thin, flexible plates - so-called \textit{large deflection theory}. These
 equations are well known in nonlinear elasticity and
constitute a basic model describing nonlinear oscillations of a
plate accounting for  large displacements, see \cite{karman}
and also \cite{springer,ciarlet,lagnese}  (and references therein).

  \par
The gas flow environment we consider is $\realsthree_+=\{(x,y,z): z > 0\}$. The plate
 is  immersed in an inviscid  flow (over body, $z\le 0$) with velocity $U \neq 1$ in the  negative $x$-direction. (Here we normalize $U=1$ to be Mach 1, i.e. $0 \le U <1$ is subsonic and $U>1$ is supersonic.) This situation, for instance, corresponds to the dynamics of a panel element of an aircraft flying with the speed $U$, see, e.g.,
 \cite{dowell1,dowellnon}.
\par
The plate is modeled by  a bounded domain $\Omega \subset \reals^2_{\{(x,y)\}}=\{(x,y,z): z = 0\}$ with smooth boundary $\partial \Omega = \Gamma$ and
the scalar function $u: \Omega \times \R_+ \to \reals$ represents the vertical displacement of the plate in the $z$-direction at the point $(x,y)$ at the moment $t$.
We focus on the plate with clamped boundary
conditions\footnote{
The clamped boundary conditions are the most physically relevant boundary conditions for the flow-plate model; additionally, clamped boundary conditions allow us to avoid certain technical issues in the consideration and streamline our exposition. Other possible and relevant plate boundary conditions in this setup include: free, hinged (or simply supported), hinged dissipation, and combinations thereof \cite{springer,lagnese,websterlasiecka}.
}.

Accepting von Karman large deflection hypotheses we arrive at the following system:
\begin{equation}\label{plate0}\begin{cases}
u_{tt}+\Delta^2u+ku_t+f(u)= p(\xb,t) ~~ \text { in } ~\Omega\times (0,T), \\
u=\Dn u = 0  ~~\text{ on } ~ \partial\Omega\times (0,T),  \\
u(0)=u_0,~~u_t(0)=u_1.
\end{cases}
\end{equation}

We take the nonlinearity to be von Karman:

\begin{equation}\label{karman-f}
  f(u)=-[u, v(u)+F_0],
\end{equation}
where $F_0$ is a given forcing function and
the von Karman bracket $[u,v]$  is given by
\begin{equation*}%\label{bracket}
[u,v] = \partial ^{2}_{x} u\partial ^{2}_y v +
\partial ^{2}_y u\partial ^{2}_{x} v -
2\partial ^{2}_{xy} u\partial ^{2}_{xy}
v,
\end{equation*} and
the Airy stress function $v(u)$ is defined  by the relation $v(u)=v(u,u)$, where $v(u,w)$
 solves the following  elliptic problem
\begin{equation}\label{airy-1}
\Delta^2 v(u,w)+[u,w] =0 ~~{\rm in}~~  \Omega,\quad \Dn v(u,w)= v(u,w) =0 ~~{\rm on}~~  \partial\Omega,
\end{equation}
for given $u,w\in H^2_0(\Om)$.
\par
For the flow component of the model, we make use of linearized
 potential theory, and we know \cite{bolotin,dowell1} that the (perturbed) flow potential $\phi:\realsthree_+ \rightarrow \reals$ must satisfy the perturbed wave equation below (note that when $U=0$ this is the standard wave equation):
\begin{equation}\label{flow}\begin{cases}
(\partial_t+U\partial_x)^2\phi=\Delta \phi & \text { in } \realsthree_+ \times (0,T),\\
\phi(0)=\phi_0;~~\phi_t(0)=\phi_1,\\
\Dn \phi = d(\xb,t)& \text{ on } \realstwo_{\{(x,y)\}} \times (0,T).
\end{cases}
\end{equation}
The strong coupling here takes place in the  downwash term of the flow potential (the Neumann boundary condition) by taking $$d(\xb,t)=-\big[(\partial_t+U\partial_x)u (\xb)\big]\cdot \mathbf{1}_{\Omega}(\xb),~~~
\xb\in\R^2,
$$
and by taking in \eqref{plate0} the aerodynamical pressure
of the form
\begin{equation}\label{aero-dyn-pr}
p(\xb,t)=p_0(\xb)+\big(\partial_t+U\partial_x\big)\gamma[\phi]~~~\mbox{with}~~ \ga[\phi]\equiv \phi\big|_{z=0}.
\end{equation}
Here above $\mathbf{1}_{\Omega}(\xb)$ denotes the indicator function of $\Om$ in $\R^2$.
 This structure of $d(\xb,t)$  corresponds  to the case when the part of boundary $z=0$
 outside of the plate is a
 the surface of a  rigid body.

This gives the fully coupled model:
\begin{equation}\label{flowplate}\begin{cases}
u_{tt}+\Delta^2u+ku_t+f(u)= p_0+\big(\partial_t+U\partial_x\big)\gamma[\phi] & \text { in } \Omega\times (0,T),\\
u(0)=u_0;~~u_t(0)=u_1,\\
u=\Dn u = 0 & \text{ on } \partial\Omega\times (0,T),\\
(\partial_t+U\partial_x)^2\phi=\Delta \phi & \text { in } \realsthree_+ \times (0,T),\\
\phi(0)=\phi_0;~~\phi_t(0)=\phi_1,\\
\Dn \phi = -\big[(\partial_t+U\partial_x)u (\xb)\big]\cdot \mathbf{1}_{\Omega}(\xb) & \text{ on } \realstwo_{\{(x,y)\}} \times (0,T).
\end{cases}
\end{equation}

In this situation, a complete description of well-posedness would require an in depth discussion of strong solutions to the (\ref{flowplate}) system, including the semigroup formulation and discussion of the generator of the dynamics. In addition, these results are not uniform with respect to the parameter value $U$. We refer the reader interested in these details to \cite{b-c,jadea12,supersonic,webster}, see also \cite{springer}
and the references therein. For this treatment, the key point is the well-posedness of weak solutions to (\ref{flowplate}). These weak solutions satisfy the variational formulation as defined in \cite{springer}.
For the purpose of this work we simply cite a recently obtained  \cite{supersonic}  well-posedness result which attests that the dynamical system generated by  (\ref{flowplate}) is associated to a strongly continuous semigroup on the phase space $$ H \equiv H^2_0 (\Omega) \times L_2(\Omega) \times H^1(\realsthree_+) \times L_2(\realsthree_+)$$
The corresponding result proved in \cite{supersonic} is stated below.
\begin{theorem}\label{well-U}
Let $ 0 \leq U \ne 1 $, $ k \ge 0 $
and $F_0 \in H^{3+\delta}(\Omega)$, $p_0 \in L_2(\Omega)$.
With reference to the system defined in  (\ref{flowplate})  and any initial data  $y_0 \equiv (u_0,u_1;\phi_0,\phi_1) \in H $, there exists a unique solution
$ y(\cdot ) \in  C([0, \infty); H ) $ which is represented by a strongly continuous semigroup $ T_t : H \rightarrow H$,
 $y(t) = T_ty_0$, $t > 0$, with the estimate
 \begin{equation*}%\label{lip}
 ||T_ty_0 - T_tz_0   ||_H \leq C(R) e^{\omega_R t } ||y(0) - z(0)||_H , ~~~ \forall\, ||y_0||_H \leq R,
  ||z_0||_H \leq R,
 \end{equation*}
 where $C(R)$ and $\om_R$ are positive constants.
  \end{theorem}
  \begin{remark}{\rm
  When $ 0 \leq U < 1 $ the semigroup $T_t$ in Theorem \ref{well-U} is stable in some extended space
  $\widetilde{H}$, i.e. there is a space $\widetilde{H}\supset H$ such that
  $||T_t y||_{\widetilde{H}} \leq C (R)$, $t > 0$.
  The above estimate (valid also for $k =0 $) owes its validity to the  nonlinear effects \cite{webster,jadea12,springer}. It is not valid for the corresponding linear semigroup.
 }
  \end{remark}

Various past considerations (se, e.g., \cite{springer}) of  systems like (\ref{flowplate}) have made use of an explicit solver for the flow. In such an approach, we may rewrite the system above as a von Karman system with delay of the form in the earlier sections. Reducing the flow-plate problem to a delayed  von Karman plate is the primary motivation for this treatment and allows long-time behavior analysis of the flow-plate system, which is considerably more difficult otherwise.
The exact  statement of this reduction is given in the following assertion:

\begin{theorem}\label{rewrite}
Let the hypotheses of Theorem~\ref{well-U} be in force, and $(u_0,u_1;\phi_0,\phi_1) \in H^2_0 (\Omega) \times L_2(\Omega) \times H^1(\realsthree_+) \times L_2(\realsthree_+)$. Assume that there exists an $R$ such that $\phi_0(\xb) = \phi_1(\xb)=0$ for $|\xb|>R$.  Then the there exists a time $t^{\#}(R,U,\Omega) > 0$ such that for all $t>t^{\#}$ the weak solution $u(t)$ to (\ref{flowplate}) satisfies the following equation:
\begin{equation}\label{reducedplate}
u_{tt}+\Delta^2u+ku_t-[u,v(u)+F_0]=p_0-(\partial_t+U\partial_x)u-q^u(t)
\end{equation}
with
\begin{equation}\label{potential}
q^u(t)=\dfrac{1}{2\pi}\int_0^{t^*}ds\int_0^{2\pi}d\theta [M^2_{\theta}\widehat u](x-(U+\sin \theta)s,y-s\cos \theta, t-s).
\end{equation}
Here, $\widehat u$ is the extension\footnote{
 This extension of the solution $u(t)$ is possible owing to the clamped boundary conditions.
}
 of $u$
   by 0 outside of $\Omega$; $M_{\theta} = \sin\theta\partial_x+\cos \theta \partial_y$ and \begin{equation}\label{delay} t^*=\inf \{ t~:~\xb(U,\theta, s) \notin \Omega \text{ for all } \xb \in \Omega, ~\theta \in [0,2\pi], \text{ and } s>t\}
\end{equation} with $\xb(U,\theta,s) = (x-(U+\sin \theta)s,y-s\cos\theta) \subset \realstwo$.
\end{theorem}
Thus, after some time, the behavior of the flow can be captured by the aerodynamical pressure term $p(t)$ in the form of a reduced delayed forcing.  This representation has been used in previous considerations of long-time behavior of plates and shallow shells \cite[Section 6.6, pp. 312-334]{springer} (and the references therein). A rigorous proof of this representation can be found in \cite[pp. 333-334]{springer} for the rotational case (when we have
additional regularity of the plate velocity $u_t$). For the reader's convenience, in Section~\ref{flowplateint} we provide a sketch of the proof, which extends the arguments given in \cite{springer} for the {\it  rotational case}. This extension is direct,
once the following  ingredients are accounted for:
(1) The full system in (\ref{flowplate}) generates strongly continuous semigroup (see Theorem \ref{well-U}), (2) the von Karman bracket is locally Lipschitz on $H$ (see Lemma \ref{l:airy-1}), and (3) the time derivative of the delayed term $ q^u_t$   is bounded on $H$  (inequality (\ref{qnegest4}) in Proposition \ref{pr:q}) .
\par
Theorem~\ref{rewrite} allows us to suppress the dependence of the problem on the flow variable $\phi$.
Here we emphasize that the structure of aerodynamical pressure \eqref{aero-dyn-pr} posited in the hypotheses leads to the velocity term  $-u_t$ on the RHS of \eqref{reducedplate}.
 One can be absorb this term into the damping coefficient $k$ on the LHS. However, since we have made no assumptions on the value of $k$, \textit{we may strengthen our result for the full reduced flow-plate system by henceforth assuming $k=0$} and utilizing the natural damping appearing in the structure of the reduced flow pressure, i.e., by moving this term to the RHS.
\par
As we see below, the reduction method above allows us to study long-time behavior of the dynamical system corresponding to (\ref{flowplate}) (for sufficiently large times) by reducing the problem to a plate equation with delay. The flow state variables $(\phi,\phi_t)$  manifest themselves in our rewritten system via the delayed  character of the problem; they appear  in the initial data for the delayed  component of the plate, namely $u^t\big|_{(-t^*,0)}$.  Hence the behavior of both dynamical systems agree for all $t>t(R,U,\Omega)$. By the dynamical systems property for the full system (see Theorem \ref{well-U}), we can propagate forward and simply study the long-time behavior of the plate with delay on the interval $(\sigma-t^*,\sigma+T]$ for $\sigma>t_{\#}$ and $T \le \infty$.
\smallskip\par
The following proposition
motivates the  hypotheses  imposed below on the delayed   force term in the von Karman plate  model (\ref{reducedplate}).

\begin{proposition}\label{pr:q}
Let $q^u(t)$ be given by (\ref{potential}). Then \begin{equation}\label{qnegest}
||q^u(t)||^2_{-1} \le ct^*\int_{t-t^*}^t||u(\tau)||^2_1d\tau
\end{equation} for any $u \in L_2(t-t^*,t;H_0^1(\Omega))$.
If $u \in L_2^{loc}([-t^*,+\infty[;H^2\cap H_0^1)(\Omega))$ we also have
\begin{equation}\label{qnegest2}
||q^u(t)||^2 \le ct^*\int_{t-t^*}^t||u(\tau)||^2_2d\tau,~~~\forall t\ge0,
\end{equation}
and
\begin{equation}\label{qnegest3}
\int_0^t ||q^u(\tau)||^2 d\tau \le c[t^*]^2\int_{-t^*}^t||u(\tau)||^2_2d\tau
,~~~\forall t\ge0.
\end{equation}
Moreover if $u \in C(-t^*,+\infty;H^2\cap H_0^1)(\Omega))$, we have that  $q^u(t) \in C^1( \R_+; H^{-1}(\Omega))$,
\begin{equation}\label{qnegest4}
\|q^u_t(t)\|_{-1} \le C\Big\{ ||u(t)||_1+||u(t-t^*)||_1+\int_{-t^*}^0||u(t+\tau)||_2d\tau\Big\},
~~~\forall t\ge0.
\end{equation}
\end{proposition}
For the proof, we refer to  Section~\ref{appendix} below.
\begin{remark}
{\rm
A priori, when $u_t$ is in $H^1_0(\Omega)$, it is clear from \eqref{qnegest} that there is a compactness margin  and
 we have the estimate
 $$
 \int_0^t<q^u(\tau),u_t(\tau)>d\tau\le
 \epsilon \int_0^t ||u_t(\tau)||^2_1+C(\epsilon,t)\sup_{\tau \in [-t^*,t]}||u(\tau)||^2_{1}.
 $$
 However, this is not immediately apparent when $u_t \in L_2(\Omega)$ as $||q^u(t)||^2_0$ has no such a priori bound from above, as in \eqref{qnegest}.
Hence, the critical component which allows us a transition from the $\gamma>0$ case (with damping of the form $k(1-\gamma \Delta)u_t$) to the $\gamma =0$ case is the hidden compactness of the aforementioned term displayed by \eqref{qnegest4}.
  We note that inequality (\ref{qnegest4}) represents a loss of one derivative  (anisotropic  - time derivatives are scaled by two spatial derivatives), versus the loss of two derivatives in  (\ref{qnegest}), (\ref{qnegest2}), and (\ref{qnegest3}).
}
\end{remark}

\subsection{PDE Description of the  Plate Model with the   Delay }
Below we utilize a positive parameter $0<t^*<+\infty$ as the time of delay, and
accept  the commonly  used (see, e.g.,   \cite{Delay-book1995} or \cite{wu-1996}) notation $u^t(\cdot)$
 for function on $s\in [-t^*,0]$ of the form $s\mapsto u(t+s)$.
We need this because of the delayed character of the problem
which requires initial data of the prehistory interval  $[-t^*,0]$, i.e.,
need to impose an initial condition
 of the form $u|_{t \in (-t^*,0)} = \eta(\xb, t)$,
 where $\eta$ is a given function on $\Om\times [-t^*,0]$.
 We can choose this prehistory data $\eta$ in different
 classes. In our problem it is convenient to deal
 with Hilbert type structures, and therefore we assume in the further considerations that
 $ \eta \in L_2(-t^*,0;H^2_0(\Omega))$. Since we do not assume
 the continuity of $\eta$ in $s\in [-t^*,0]$,  we also
 need to add the (standard) initial conditions of the form
 $u(t=0)= u_0(\xb)$ and $\partial_t u(t=0)=u_1(x)$.
 \par
 Again, employing von Karman large deflection hypotheses we arrive at the following system:
\begin{equation}\label{plate}\begin{cases}
u_{tt}+\Delta^2u+ku_t+f(u) +Lu= p_0+q(u^t,t) ~~ \text { in } ~\Omega\times (0,T), \\
u=\Dn u = 0  ~~\text{ on } ~ \partial\Omega\times (0,T),  \\
u(0)=u_0,~~u_t(0)=u_1,~~\\ u|_{t \in (-t^*,0)} = \eta\in L_2(-t^*,0;H^2_0(\Omega)).
\end{cases}
\end{equation}
Here $f(u)$  is given by \eqref{karman-f}.
The forcing term $q(u^t,t)$ occurring on the RHS of the plate equation will encompass the delayed potential
of the gas flow
and given by the function $q :\, L_2(-t^*,0;H^2_0(\Omega))\times \R\mapsto \R$, which will be specified below.
 The scalar $k\ge 0$ is our damping coefficient, and represents constant viscous damping across the full interior of the plate. The operator $L$ encompasses spatial lower order terms which do not have gradient structure
 (e.g., the term $-Uu_x$  in \eqref{reducedplate}).

\begin{remark}\label{rotational}
{\rm As it was already mentioned above,
the  basic plate model we consider  may include a rotational inertia term (see, e.g.
\cite{lagnese} or \cite{springer}), corresponding to the parameter $\gamma \ge 0$ and accompanying damping parameters $k_1,k_2 >0$. This leads to a plate equation of the form
$$
(1-\gamma\Delta)u_{tt}+\Delta^2u+(k_1-k_2\Delta)u_t+f(u)+Lu=p_0+q(u^t,t).
$$
Recall that the parameter $\gamma$ corresponds to rotational inertia in the filaments of the plate, as discussed in the Introduction.
These kind of delay models were studied in  and \cite{Chu92b} and \cite[Sections 3.3.1 and 9.3.1]{springer}.
We also note that in the case when
\[
f(u)=f_0\left(\int_\Om|\nabla u(\xb)|^2d\xb\right)
\]
in \eqref{plate}, with an appropriate $C^1$ function $f_0$,
we arrive to the Berger plate model with delay which was studied in
\cite{oldchueshov1,oldchueshov2}.
}
\end{remark}
Now we formulate our standing hypotheses; we begin with those responsible for well-posedness of the model in \eqref{plate}:

\begin{assumption}\label{as:1}
\begin{itemize}
\item  We suppose  $f(u)=-[u, v(u)+F_0]$,  where the functions $F_0$ and $p_0$ possess the properties:
$$
F_0(\xb) \in H^{3+\delta}(\Omega)~~\mbox{for some $\delta>0$},~ ~p_0(\xb) \in L_2(\Omega).
$$
\item The linear operator $L: H_0^2(\Omega) \to L_2(\Omega)$ is continuous.
\item $v\mapsto q (v, t)$ is a continuous linear mapping from $L_2(-t^*,0;H^2(\Omega))\times\R_+$ to $L_2(\Omega)$ possessing the property:
\begin{equation}\label{assumpt1-0}
||q(u^t, t)||^2 \le~ C\int_{t-t^*}^t||u(\tau)||_2^2 d\tau, ~~~\forall\, t\ge 0,
~~\forall\, u\in L_2^{loc}([-t^*,+\infty[;H^2(\Omega)).
\end{equation}

\end{itemize}
\end{assumption}

Additional hypotheses are needed for long-time dynamics:

 \begin{assumption}\label{as:2}
 \begin{itemize}
\item The linear operator $L: H_0^{2-\delta}(\Omega) \to L_2(\Omega)$ is continuous for some $\delta>0$.
\item $q(v,t)$  possesses the (additional) property:
\begin{equation}\label{assumpt1}
||q(u^t,t)||_{-\sigma}^2 \le~ C\int_{t-t^*}^t||u(\tau)||_{2-\sigma}^2 d\tau~~~\mbox{for some $0<\sigma<2$,}
\end{equation}
with any $t>0$ and   $u\in L_2^{loc}([-t^*,+\infty[;H^2(\Omega))$.

\item
We assume that the generalized time derivative $\partial_t [q(u^t,t)]$ belongs to $H^{-2}(\Omega)$ a.s.
for any $u \in C(-t^*,T;H_0^2(\Omega))$  with the following estimate holding for any $\psi \in H_0^2(\Omega)$:
\begin{equation}\label{assumpt1.75}
|<\partial_t [q(u^t,t)],\psi>| \le C\left(||u(t)||_{2}+||u(t-t^*)||_{2}+\int_{-t^*}^0||u(t+\tau)||_2 d\tau\right) ||\psi||_{2-\eta}
\end{equation} for some $\eta>0$.
 \end{itemize}
 \end{assumption}
 \begin{remark}\label{re:as12}
 {\rm

By Proposition~\ref{pr:q}  $q(u^t,t)\equiv q^u(t)$ given by \eqref{potential}
satisfies both  Assumptions \ref{as:1} and \ref{as:2}.
Roughly speaking,
the conditions in \eqref{assumpt1-0}--\eqref{assumpt1.75}  mean that
the delay time of the system is distributed in the interval $[-t^*,0]$, with
density which is absolutely continuous with respect to Lebesgue measure.
This observation also implies that
a delay term of the form
\[
q(v,t)=\int_{-t^*}^0 Q(t,\tau)v(\tau)d\tau,
\]
where $Q(t,\tau)$ is a family of linear operators from $H_0^2(\Om)$ into $L_2(\Om)$
satisfying appropriate hypotheses, could be studies from the point of view of this treatment.
We also note
that we will use the  estimate in  \eqref{assumpt1.75} to derive a result on `hidden' compactness of the term $$\int_0^t<q(u^\tau,\tau),u_t(\tau)>d\tau,$$ which is arrived at via integration by parts in time.
 }
  \end{remark}

\subsection{Well-Posedness of the Plate Model and Energy Relation}
Long-time behavior analysis of the delayed system depends on  the well-posedness of suitably
defined weak solutions which generate a dynamical system on the  phase space $\bH = H_0^2(\Omega) \times L_2(\Omega) \times L_2(-t^*,0;H_0^2(\Omega))$.

 Well-posedness of weak solutions  has been addressed \cite{Chu92b} and
 \cite[Section 3.3.1, pp. 189-192; 221-222]{springer}
via the Galerkin method, see also \cite{oldchueshov1,oldchueshov2}
in the case of Berger plates. In what follows we  summarize and complement relevant results.
\par

 We take a \textit{weak solution} to \eqref{plate} on $[0,T]$ to be a function $$u \in L_{\infty}(0,T;H_0^2(\Omega))\cap W^1_{\infty}(0,T;L_2(\Omega)) \cap L_2(-t^*,0;H_0^2(\Omega))$$ such that the variational relation corresponding to \eqref{plate} holds (see, e.g., \cite[(4.1.39), p.211]{springer}).
We now assert:
\begin{proposition}\label{p:well}
Let Assumptions~\ref{as:1} be in force. Then, with initial data
$$
(u_0,u_1,\eta) \in\bH\equiv H_0^2(\Omega)\times L_2(\Omega) \times L_2(-t^*,0;H_0^2(\Omega)),
$$
problem (\ref{plate}) has a unique weak solution on $[0,T]$ for any $T>0$. This solution belongs to
the class
$$C(0,T;H_0^2(\Omega))\cap C^1(0,T;L_2(\Omega))
$$
and satisfies the energy identity
\begin{equation}\label{energyrelation}
\cE(t)+k\int_s^t ||u_t(\tau)||^2d\tau=\cE(s)+
\int_s^t<q(u^{\tau},\tau),u_t(\tau)>d\tau+\int_s^t<p_0-Lu(\tau),u_t(\tau)>d\tau,
\end{equation}
where the full (not necessarily positive) energy has the form
\begin{equation}\label{fulle}\cE(u,u_t) \equiv \dfrac{1}{2}\big\{||u_t||^2-<[u,F_0],u>\big\}+\Pi_*(u)
\end{equation}
with \begin{equation}\label{potentiale}
\Pi_*(u) \equiv \dfrac{1}{2}\big\{ ||\Delta u||^2+\dfrac{1}{2}||\Delta v(u)||^2\big\}.
\end{equation}
\end{proposition}

As stated in \cite[Section 4.1.6, p.221]{springer},
the proof of Proposition \ref{p:well} requires only minor modifications  with respect to the proof of the related result in
 Theorem 3.1.1 \cite{springer} on p.190.
 Since Theorem 3.1.1 deals with rotational models ($\gamma > 0 $), in order to handle the effect of nonlinear term we
 rely on the {\it sharp regularity} of Airy's stress function $v(u)$, given below:
 \begin{lemma}\label{l:airy-1}
 The function $v(u,w)$ defined in (\ref{airy-1}) satisfies
 \begin{enumerate}
 \item
 $|v(u,w)|_{W^{2,\infty}(\Omega) } \leq C ||u||_2||w||_2 $.
 \item
 The map $u,w \rightarrow v(u,w) $ is locally Lipschitz from
 $H_0^2(\Omega) \times H_0^2(\Omega) \rightarrow W^{2,\infty}(\Omega)$.
 \end{enumerate}
 \end{lemma}
 This lemma easily implies that $f(u) $ in (\ref{plate}) given by
 $f(u) = - [u, v(u) + F_0 ] $ is locally Lipschitz.

 Since the topology for the velocity $u_t$ is now $L_2(\Omega) $, (rather than $H_0^1(\Omega)$, as in the rotational
 case $\ga>0$), we use inequality (\ref{assumpt1-0}) rather than (3.3.5) p. 189 \cite{springer}.
 This modification allows us to repeat the arguments in \cite{springer}, in order to conclude with the statements of  Proposition~\ref{p:well}. For this we also make use of the first part of the following lemma:

 \begin{lemma}\label{le:q} We denote $q^u(t)=q(u^t,t)$.
 Let Assumption~\ref{as:1} be in force. Then
\begin{align}\label{hidden1}
\Big|\int_0^t <q^u(\tau),u_t(\tau)> d\tau\Big| \le&~   C\e^{-1} t^*  \int_{-t^*}^t||u(\tau)||_2^2d\tau +\e \int_0^t ||u_t(\tau)||^2d\tau,~~~\forall \e>0,~\forall t\in[0,T],
\end{align}
for any  $u \in L_2(-t^*,T;H^2(\Omega))\cap W^1_2(0,T;L_2(\Omega))$.
\par
If, in addition, we assume Assumption~\ref{as:2}, then there exists $\eta_*>0$ such that for every $\epsilon>0$  we have the estimate:
\begin{align}\label{hidden2} \Big|\int_0^t <q^u(\tau),u_t(\tau)> d\tau\Big| \le& ~ \epsilon\int_{-t^*}^t ||u(\tau)||_{2}^2d\tau +C(t^*,\e)\cdot(1+T)\sup_{[0,t]}||u(\tau)||^2_{2-\eta_*},~~\forall t\in[0,T],
\end{align}
 for  any $u \in L_2(-t^*,T;H^2(\Omega))\cap C(0,T;H^{2-\eta_*}(\Omega))\cap C^1(0,T;L_2(\Omega))$.
 \end{lemma}
 For the proof we refer to Section~\ref{appendix}.
 \smallskip\par

 In order to consider the delayed system as a dynamical system with the phase space $\mathbf{H}$ we recall  the notation:
 $u^t(s) \equiv u(t+s) , s \in [-t^*, 0 ]$.
 With the above notation we introduce the  operator
$ S_t\, : \bH\mapsto \bH$ by the formula
\begin{equation}\label{semigroup}
  S_t(u_0, u_1, \eta) \equiv (u(t), u_t(t), u^t),
\end{equation}
where $u(t)$ solves \eqref{plate}.
   Proposition \ref{p:well} implies the following conclusion
 \begin{corollary}\label{co:generation}
 $S_t : \mathbf{H} \rightarrow \mathbf{H}  $ is a strongly continuous semigroup on $\mathbf{H}$.
 \end{corollary}
\begin{proof}
Strong continuity is stated in Proposition~\ref{p:well}.
 The semigroup property  follows from uniqueness.
To prove continuity with respect to initial data we use the following assertion.

\begin{lemma}\label{l:lip}
Suppose $u^i(t)$ for $i=1,2$ are weak solutions to (\ref{plate}) with different initial data and $z=u^1-u^2$. Additionally assume that
\begin{equation}\label{bnd-R}
||u_t^i(t)||^2+||\Delta u^i(t)||^2 \le R^2, ~i=1,2
\end{equation}
 for some $R>0$ and all $t \in [0,T]$. Then there exists $C>0$ and $a_R\equiv a_R(t^*)>0$  such that
\begin{align}\label{dynsysest}
||z_t(t)||^2+||\Delta z(t)||^2 \le &~ Ce^{a_Rt}\Big\{||\Delta(u^1_0-u^2_0)||^2+||u^1_1-u^2_1||^2+\int_{-t^*}^0||\eta^1(\tau)-\eta^2(\tau)||_2^2d\tau\Big\}
\end{align}
for all $t \in [0,T]$.
\end{lemma}
\begin{proof}
We have that $z$ solves the following problem
\begin{equation}\label{difference}
\begin{cases}z_{tt}+\Delta ^2 z + k z_t+f(u^1)-f(u^2)=q^z(t)-Lz,
 \\ z=\Dn z = 0 \text{ on $\partial \Omega$ },
 \\ z(0)=z_0 \in H^2_0(\Omega), ~z_t(0)=z_1 \in L_2(\Omega), ~z|_{(-t^*,0)}\in L_2(-t^*,0;H_0^2(\Omega)),
 \end{cases}
\end{equation}
where as above we denote $q^z(t)=q(z^t,t)$.
Let
\begin{equation}\label{Ez}
E_z(t) \equiv \dfrac{1}{2}\big\{||\Delta z(t)||^2 + ||z_t(t)||^2\big\},
\end{equation}
then
making use of the the energy equality for the difference  $z$
\begin{align*}%\label{zenergyrelation-00}
\Ez(t)+k\int_0^t ||z_t||^2 d\tau =& ~\Ez(0) -\int_0^t <f(u^1)-f(u^2),z_t>d\tau +\int_0^t<q^z(\tau),z_t(\tau)> d\tau
\\ \notag
&
-\int_0^t<Lz(\tau),z_t(\tau)>d\tau
\end{align*}
 we have the following:
\begin{equation*}
E_z(t) + k\int_0^t ||u_t||^2 d\tau \le E_z(0)+\int_0^t\left[\|f(u^1)-f(u^2)\|+ \| L z\|\right] \|z_t\| d\tau | + \Big|\int_0^t<q^z,z_t>d\tau  \Big|.
\end{equation*}
From here, we apply the estimates in (\ref{hidden1}) for $q^z$ with $\epsilon=1/t^*$, and note the locally Lipschitz character of the von Karman nonlinearity $f(u)$, yielding
\begin{equation*}
E_z(t) + k\int_0^t ||u_t||^2 d\tau \le C_0\Big(E_z(0)+\int_{-t^*}^0 ||z(\tau)||^2_2d\tau\Big)+C(R,t^*)\int_0^t E_z(\tau) d\tau.
\end{equation*}
Hence, Gronwall's inequality yields the desired estimate in \eqref{dynsysest}.
\end{proof}
Using Lemma~\ref{l:lip} we obtain that
\begin{equation*}
||S_{t}y_{1}-S_{t}y_{2}||_{\bH}^{2}\leq C e^{a_Rt}||y_{1}-y_{2}||_{\bH}^{2}
\end{equation*}
for $S_ty_i=(u^i(t), u^i_t(t), [u^i]^t)$ with $y_i=(u^i_0, u^i_1, \eta)$,
such that \eqref{bnd-R} holds.
This allows us to conclude the proof of Corollary~\ref{co:generation}.
\end{proof}

\section{Statement of Main Results}\label{Sect3}
Our main results deal with (1) long-time dynamics of the system $(S_t,\bH)$ generated by \eqref{plate} and  (2) its connection with  the flow-structure dynamics governed by \eqref{flowplate}.

\begin{theorem}\label{maintheorem}
Let   both Assumptions \ref{as:1} and \ref{as:2} be in force
and  $(S_t,\bH)$ be the dynamical system generated
weak solutions   to the system in \eqref{plate} with $k > 0 $  on the space $\bH \equiv H_0^2(\Omega)\times L_2(\Omega)\times L_2(-t^*,0;H_0^2(\Omega))$ with evolution operator given by \eqref{semigroup}.
 Then the  system $(S_t,\bH)$ has a compact global attractor $\bA$ of finite fractal dimension. The attractor can be characterized as the set of all bounded full trajectories. Moreover, the set $\bA$ has additional regularity; namely, any full trajectory $y(t)=(u(t),u_t(t),u^t)\subset \bA$, $t\in\R$, has the property that $u \in L_{\infty}(\R;H^4(\Omega)\cap H_0^2(\Omega))$ and $u_t \in L_{\infty}(\R;H^2(\Omega))$.
\end{theorem}
We recall (see, e.g., \cite{Babin-Vishik,ch-0,lad,temam})   that
a \textit{global attractor} $\mathbf{A}$ is a closed, bounded set in
$\bH$ which is invariant (i.e., $S_t\mathbf{A}=\mathbf{A}$ for all
$t>0$) and uniformly attracts every bounded set $B$, i.e.
\begin{equation}  \label{dist-u}
\lim_{t\to+\infty}d_{\bH}\{S_t B|\bA\}=0,~~~\mbox{where}~~
d_{\bH}\{S_t B|\bA\}\equiv\sup_{y\in B}{\rm dist}_{\bH}(y,\bA),
\end{equation}
for any bounded  $B\in\bH$.
\par
The {\it fractal} (box-counting) {\it dimension} ${\rm dim}_f\bA$ of $\bA$ is defined by
\[
{\rm dim}_f\bA=\limsup_{\e\to 0}\frac{\ln n(\bA,\e)}{\ln (1/\e)}\;,
\]
where $n(M,\e)$ is the minimal number of closed balls in $\bH$ of the
radius $\e$ which cover the set $M$.

\begin{remark}\label{notgamma}
{\rm
We note that this type of additional regularity of solutions from the attractor mentioned in Theorem~\ref{maintheorem} is not possible in the case $\gamma>0$, owing to the fact that the principal term in the equation is $\Delta^2u-\gamma \Delta u_{tt}$. In this case, the presence of this term disallows the use of elliptic regularity theory (applied to the biharmonic term) on elements of the attractor. More importantly, when $\gamma > 0 $,  in order to obtain the strong attractiveness property a much stronger damping is necessary.
In order to stabilize the kinetic  part of the energy one will have to introduce $-\gamma \Delta u_t$ as a damping term (see section 9.3 in \cite{springer}). The point we want to stress is that in our case only $u_t$ as a damping term is needed.
And, in fact, it is this term that will be generated by the flow from ``thin air". As a consequence, the plate (in the full flow-plate system) will require no mechanical damping at all.
}
\end{remark}

Having established the quasicompact character of the delayed  potential $q^u$ as in (\ref{potential}) (showing that it satisfies the conditions (\ref{assumpt1-0})--(\ref{assumpt1.75}))  we can now apply  Theorem \ref{maintheorem}
 to the von Karman flow-plate model in \eqref{flowplate}.
\begin{theorem}\label{th:main2}
Suppose $0\le U \ne 1$, $k\ge 0$, $F_0 \in H^{3+\delta}(\Omega)$ and $p_0 \in L_2(\Omega)$.
 Then there exists a compact set $\mathscr{U} \subset H_0^2(\Omega) \times L_2(\Omega)=\mathcal H$ of finite fractal dimension such that $$\lim_{t\to\infty} d_{\mathcal H} \big( (u(t),u_t(t)),\mathscr U\big)=\lim_{t \to \infty}\inf_{(\nu_0,\nu_1) \in \mathscr U} \big( ||u(t)-\nu_0||_2^2+||u_t(t)-\nu_1||^2\big)=0$$
for any weak solution $(u,u_t;\phi,\phi_t)$ to (\ref{flowplate}) with
initial data
$$
(u_0, u_1;\phi_0,\phi_1) \in H\equiv H_0^2(\Omega)\times L_2(\Omega)\times H^1(\realsthree_+)\times L_2(\realsthree_+)
$$
which are
localized  in $\R_+^3$ (i.e., $\phi_0(\xb)=\phi_1(\xb)=0$ for $|\xb|>R$ for some $R>0$). Additionally, we have the additional regularity $\mathscr{U} \subset \big(H^4(\Omega)\cap H_0^2(\Omega)\big) \times H^2(\Omega)$.
\end{theorem}
\begin{proof}
The proof of this result follows from rewriting the  dynamical system  $(T_t, H )$ generated by    (\ref{flowplate}) as the delayed  system in (\ref{reducedplate}). The latter  is possible for sufficiently large times by Theorem \ref{rewrite}. Since the delayed  potential $q^u$ was shown to satisfy Assumptions~\ref{as:1} and \ref{as:2}, we may apply our main result, Theorem \ref{maintheorem}, to the dynamical system generated by the weak solution to (\ref{reducedplate}) on the space $\bH=H_0^2(\Omega)\times L_2(\Omega) \times L_2(-t^*,0;H_0^2(\Omega))$. This yields a compact global attractor $\mathbf A \subset \bH$ of finite dimension  and additional regularity; we then take $\mathscr U$ to be the projection of $\mathbf{A}$ on $H$, which concludes the proof  as in  \cite{springer}.
\end{proof}
\begin{remark}
{\rm
We here reiterate that the above result holds in the absence of imposed damping, i.e., with $k=0$ in
 \eqref{flowplate}. Utilizing the natural damping appearing in the reduced flow pressure, we see that in the case of $\gamma = 0$,  {\it the flow naturally provides a stabilizing effect} to the dynamics in that it yields the existence of the compact attractor. This is \textit{not} the case when $\gamma >0$, as the nature of the damping must be (necessarily) stronger (see Remark \ref{notgamma}).}
\end{remark}
\par
\begin{remark}
{\rm
It should also be noted here that because we have rewritten our problem \eqref{flowplate} as a reduced delayed plate, and additionally changed the state space upon which we are operating, the  results obtained  on long-time behavior \textit{ will not be invariant } with respect to the flow component of the model, i.e. our global attractors will be with respect to the state space $\bH$, as defined above. Again, the data in the form of the delayed  term $u|_{(-t^*,0)}$
contains the information  from the flow itself. Obtaining global attractors for the full state space corresponding to $(u,u_t;\phi,\phi_t) \in H^2_0(\Omega) \times L_2(\Omega) \times H^1(\realsthree_+)\times L_2(\realsthree_+)$ is not  a
realistic task from  the  mathematical point of view. There is no damping imposed  on the system, thus the flow component
evolves according to  the full half space, unconstrained dynamics. The obtained result on the structure (without damping)  is the best  possible result with respect to both the underlying physics and mathematics of the problem.
}
\end{remark}
\smallskip\par\noindent
{\bf Previous Literature and New Challenges:}
 Nonlinear PDEs with delays  have been considered in various sources (see \cite{wu-1996} and references therein). In  relation to plate equations with delayed  aerodynamical type pressure, we note that \cite{springer} provides a rather complete analysis of the delayed von Karman plate in the presence of rotational terms (and application to flow-plate interactions), and the aforementioned references \cite{oldchueshov1,oldchueshov2} deal with the plate with delay in the presence of the Berger nonlinearity.
In this latter references, it is assumed that the parameter in front of delayed term
is {\it suitably small}.
 The analysis in the more recent reference \cite{springer} also, by and large, applies to the Berger plate
 (in the rotational case).
\par
One should stress at the outset that  the problem is challenging, even in the rotational case. This is due to the fact that the underlying system is intrinsically non-gradient (both the delay term and the term $U\partial_x u$
provide non-conservative and non-dissipative terms that contribute to the loss of gradient structure). In view of this, the existence of attractors requires a priori information on a uniformly absorbing set.
The presence of delay terms along with non-conservative terms makes the latter task challenging \cite{oldchueshov1}.

The references pertaining to plates with delayed terms  primarily  with the rotational case, i.e. the plate equation discussed in Remark \ref{rotational}. In particular, for $\gamma,k_1,k_2>0$, well-posedness of weak solutions is established via the Galerkin method \cite[Sections 3.3.1, 4.1.6]{springer}. These solution are shown to generate a dynamical system the state space $\bH \equiv H_0^2(\Omega) \times H_0^1(\Omega) \times L_2(-t^*,0;H_0^2(\Omega))$. Then, exploiting the compactness of the term $<q^u,u_t>$, i.e. making use of the duality pairing $(H_0^{1}, H^{-1})$, dissipativity of the dynamical system can be shown (in much the same way which we utilize below), followed by asymptotic smoothness. In this case, however, asymptotic smoothness is arrived at in a straightforward way, which additionally exploits the compactness of the von Karman nonlinearity (with respect to the energy identity) in the case where $u_t$ in $H^1_0(\Omega)$.

The  mathematical hurdles arising in the $\gamma=0$ case (where $u_t \in L_2(\Omega)$)  begin at the outset
with well-posedness of weak solutions; indeed, many well-posedness and long-time behavior analysis \cite{webster,websterlasiecka,supersonic} are dramatically complicated when $\gamma=0$.
Thus, it is no wonder that,  to date, the non-rotational von Karman plate with delayed terms  has not been considered. The reasons for this are clear:
(i) the methods of studying long-time behavior via an approach making use of the combination of
 Theorems \ref{dissmooth} and \ref{psi} (stated below) is relatively recent. In particular, dealing with the von Karman nonlinearity is especially difficult outside of the use of Theorem \ref{psi}. (ii) In addition, the initial studies
 \cite{Chu92b}
 of delayed von Karman plates took place before results on the sharp regularity of the Airy stress function were available. These results are critical in this treatment. (iii) Lastly, it is clear by inspection that  additional properties  of the delayed  force in (\ref{plate}) must be accounted for  when the rotational inertia term is absent. However, owing to a gap in well-posedness results for flow-plate interactions, it was unclear what these  properties - translated into abstract assumptions -
should be. A recent observation about the `hidden compactness' of the reduced delayed potential (derived from an inviscid potential flow) yielded insight into what assumptions are reasonable in line with previous analysis of von Karman plates with delay.

In proving finite-dimensionality and smoothness of the attractors, the criticality of the nonlinearity and the lack of gradient structure prevents one from using
a powerful technique of {\it backward smoothness of trajectories}  \cite{Babin-Vishik,kh,springer}, where smoothness is propagated forward from the equilibria.  Since the attractor may have complicated structure, the structure of the attractor is not characterized by the equilibria points. In order to cope with this issue, we take the advantage
of novel method that is based on density and exploits only the compactness of the attractor.

\section{Proof of Main Result}\label{Sect4}
We first  {\bf outline the steps} utilized to obtain the main result stated above.
\begin{itemize}
\item
We begin with recalling the key results on dissipative long-time
dynamics for non-gradient systems.
\item We then
 recall (in the form adapted  to the delay case)
 cite the primary estimates which have been used in previous long-time behavior considerations for von Karman plates in the past. Noting that we cannot make use of the $\gamma>0$ approach to long-time behavior (owing to the loss of compactness of the term $\ds <q^u(t),u_t(t)>$). We begin by exploiting a different assumption (`hidden' compactness of this term when integrated in $t$ in the energy relation (\ref{energyrelation})).
\item We then use a similar, modified Lyapunov functional as that in \cite[p.480]{springer} on the dynamical system to show that it is dissipative.
\item After obtaining the necessary compactness estimates, we synthesize them to produce a pointwise energy estimate which allows us to make use the abstract Theorem \ref{psi} to obtain asymptotic smoothness of the dynamical system associated with weak solutions to (\ref{plate}).
 At this point, we make use of abstract Theorem \ref{dissmooth} to conclude that the dynamical system possesses a compact attractor in the space $\bH$.
\item In the last step, we revisit our estimation in the asymptotic smoothness section to obtain the so-called quasistability estimate on the attractor utilizing its compactness; this allows us to apply Theorem \ref{t:FD}. The application of this theorem gives the finite dimensionality and additional smoothness of the attractor.
\end{itemize}

\subsection{Preliminaries on Dissipative  Dynamical Systems}
We recall notions and results from the theory of dynamical systems
 (see, e.g., \cite{Babin-Vishik,ch-0,lad,temam}).
\par
One says that a dynamical system
$(S_t,\bH)$ is \textit{asymptotically smooth} if for any
bounded, forward invariant  set $D$ there exists a compact set $K
\subset \overline{D}$ such that
\begin{equation*} % \label{dist}
\lim_{t\to+\infty}d_{\bH}\{S_t D|K\}=0
\end{equation*}
 holds. An asymptotically smooth
dynamical system should be thought of as one which possesses \textit{local
attractors}, i.e. for a given forward invariant set $B_R$ of diameter $R$ in the space $
\bH$ there exists a compact attracting set in the closure of $B_R$,
however, this set need not be uniform with respect to $R$.
\par
A closed set $B \subset \bH$ is said to be \textit{absorbing} for $(S_t,\bH)$ if for any bounded set $D \subset \bH$ there exists a $t_0(D)$ such that $S_tD \subset B$ for all $t > t_0$. If the dynamical system $(S_t,\bH)$ has a bounded absorbing set it is said to be \textit{dissipative}.
\par
In the context of this paper we will use a few keys theorems (which we now
formally state) to prove the existence of a finite dimensional global attractor.  First, we address
the existence of attractors and characterize the attracting set:

\begin{theorem}
\label{dissmooth} Any dissipative and asymptotically smooth dynamical system $(S_t,\bH)$ in a Banach space $\bH$ possesses a unique compact global attractor $\textbf{A}$. This attractor is a connected set and can be described as a set of all bounded full trajectories.
\end{theorem}
For the proof, see  \cite{Babin-Vishik} or \cite{temam}.
\par

Secondly, we state a useful criterion (inspired by \cite{kh} and proven in \cite{ch-l}, see also
\cite[Chapter 7]{springer}) which reduces
asymptotic smoothness to finding a suitable functional on the state space
with a compensated compactness condition:
\begin{theorem} \label{psi}
Let $(S_t,\bH)$ be a dynamical system. Assume that for any bounded positively
invariant set $B \subset \bH$ and for all $\epsilon>0$ there exists
a $T\equiv T_{\epsilon,B}$ such that
\begin{equation*}
||S_Tx_1 - S_Tx_2||_{\bH} \le
\epsilon+\Psi_{\epsilon,B,T}(x_1,x_2),~~x_i \in B
\end{equation*}
with $\Psi$ a functional defined on $B \times B$ depending on $\epsilon, T,$
and $B$ such that
\begin{equation*}
\liminf_m \liminf_n \Psi_{\epsilon,T,B}(x_m,x_n) = 0
\end{equation*}
for every sequence $\{x_n\}\subset B$. Then $(S_t,\bH)$ is an
asymptotically smooth dynamical system.
\end{theorem}

In order to establish both smoothness of the attractor and finite
dimensionality, a stronger estimate on the difference of two flows is needed. We now cite \cite[pp. 381-387]{springer}
and also \cite{cl-hcte}. Note that we have used a specialization of the cited theorem which utilizes the special structure of the state space in the problem at hand: $\bH = H_0^2(\Omega) \times L_2(\Omega)\times L_2(-t^*,0;H_0^2(\Omega))$); the theorems cited above are more general:

\begin{theorem}\label{t:FD}
 Let $\mathbf{A}$
be a global attractor for $(S_t,\bH)$.

$x_{1},x_{2}\in B\subset \bH$ where $B$ is a
forward invariant set for the flows $S_{t}x_{i} $.
%=(u^{i}(t);u_{t}^{i}(t))$,
Assume that the following inequality holds for all $t>0$ with positive
constants $C_{1},C_{2},\omega$
\begin{equation}
||S_{t}x_{1}-S_{t}x_{2}||_{\bH}^{2}\leq C_{1}e^{-\omega
_{B}t}||x_{1}-x_{2}||_{\bH}^{2}+C_{2}\max_{\tau \in \lbrack
0,t]}|| u_1(\tau)-u_2 (\tau) ||_{H_1}^{2}  \label{quasi}
\end{equation}
for any $x_1,x_2\in \mathbf{A}$,
where $S_tx_i=(u_i(t), \partial_t u_i(t), u_i^t)$
and
$H^2_0(\Om)\subset H_1 \subset L_2(\Om)$ is compactly embedded. Then the
attractor $\mathbf{A}$  posesses the
following properties: \newline
(a) The fractal dimension of $\mathbf{A}$ is finite. \newline
(b) For any $x \in \mathbf{A} $ one has $\partial_t\big( S_t x \big)\in
L_{\infty}(\reals, \bH) $.
% $\mathbf{A}\subset H^{4}(\Omega )\times H^{2}(\Omega )$ and $\mathcal{A}$
%is bounded with respect to the topology $H^{4}\times H^{2}$.
\end{theorem}
The estimate in (\ref{quasi}) is often referred to (in practice) as a ``quasistability"
estimate. It reflects the fact that the flow can be stabilized exponentially
to a compact set. Alternatively, we might say that the flow is exponentially
stable, modulo a compact perturbation (lower order terms). The quadratic nature of the lower order terms is important for the validity of
Theorem \ref{t:FD}.

\subsection{Technical Preliminaries}
In this section we derive and cite certain energy and multiplier estimates, as well as estimates on the von Karman nonlinearity, which will be necessary in the proof of Theorem \ref{maintheorem} below.
\smallskip\par
The following theorem is a case specialization found in \cite[Section 1.4, pp.38-45; Section 9.4, pp.496-497]{springer}. These bounds elucidate the local Lipschitz (quasi-Lipschitz) character of the von Karman nonlinearity are relatively recent and critical to our nonlinear analysis.
\begin{theorem}\label{nonest}
Let $u^i \in \mathscr{B}_R(H^2_0(\Omega))$, $i=1,2$, and $z=u^1-u^2$.
 Then for $f(u) = - [u,v(u)+F_0]$ we have
\begin{equation}\label{f-est-lip}
||f(u^1)-f(u^2)||_{-\delta}  \le C_{\delta}\big(1+||u^1||_2^2+||u^2||_2^2\big)||z||_{2-\delta} \le C(\delta,R)||z||_{2-\delta}~~~ for~ all ~~\delta \in [0,1].
\end{equation}
If we further assume that $u^i \in C(s,t;H^2(\Omega))\cap C^1(s,t;L_2(\Omega))$, then we have  that
\begin{equation*}%\label{4.9}
- <f(u^1)-f(u^2),z_{t}> =\dfrac{1}{4}\frac{d}{dt}Q(z)+\frac{1}{2}%
P(z)
\end{equation*}%
where
\begin{equation*}
Q(z)=<v(u^1)+v(u^2),[z,z]> -||\Delta v(u^1+u^2,z)||^2
\end{equation*}%
and
\begin{equation}\label{4.9aa}
P(z)=-<u^1_{t},[u^1,v(z)]> -<u^2_{t},[u^2,v(z)]> -<
u^1_{t}+u^2_{t},[z,v(u^1+u^2,z)]> .
\end{equation}
Moreover,
\begin{align}\label{eq4.5}
\Big|\int_s^t <f(u^1(\tau))-f(u^2(\tau)),z_t(\tau)>d\tau\Big| \le &~C(R)\sup_{\tau \in [s,t]} ||z||^2_{2-\eta}+\frac12 \Big|\int_s^tP(z)d\tau\Big|
\end{align} for some $0<\eta<1/2$ provided
 $u^i(\tau) \in \mathscr{B}_R(H^2_0(\Omega))$ for all $\tau\in [s,t].$
\end{theorem}
The above bounds rely on the equation
\[
f(u^1)-f(u^2) = [z,v(u)+F_0]+[u^2,v(u^1)-v(u^2)]
\]
and on the so-called `sharp' regularity of the Airy stress function $v(u)$
(see Lemma~\ref{l:airy-1}).
\smallskip\par

We will now make use of the above estimates in producing energy type estimates.

 First, we multiply (\ref{plate}) by $u$ and integrate over the set $\Omega \times (s,t)$, making use of clamped boundary conditions. This produces the following identity:
\begin{align*}
<u_t(\tau), u(\tau)>\Big|_{s}^t+ &
\int_s^t\Big( ||\Delta u(\tau)||^2  -  ||u_t(\tau)||^2 \Big)d\tau
\\ & =
-\int_s^t\Big(<f(u(\tau)),u(\tau)> + <q^u(\tau)+p_0-Lu,u(\tau)>\Big)d\tau,
\end{align*}
where as above we use the notation $q^u(\tau)=q(u^\tau,\tau)$.
By standard splitting and interpolation, we arrive at
\begin{align*}
\int_s^t\Big( ||\Delta u(\tau)||^2  -  ||u_t(\tau)||^2 \Big)d\tau  \le &~ \epsilon \int_s^t ||u||_2^2d\tau+
\e\|p+0\|^2 -\int_s^t <f(u(\tau)),u(\tau)>d\tau\\&+C\int_s^t ||q^u(\tau)||^2_{-\s}d\tau + C_\epsilon(1+|t-s|)\sup_{\tau \in [s,t]} ||u(\tau)||_{2-\eta}^2
\\ & + |<u_t(t), u(t)>| +|<u_t(s), u(s)>|
\end{align*} for all $\e>0$ and for some $\eta > 0$.
This estimate, coupled with the estimates in (\ref{assumpt1}) and (\ref{f-est-lip}) yield the following estimates:
\begin{lemma}
Let $u^i \in C(0,T;H_0^2(\Omega))\cap C^1(0,T;L_2(\Omega)) \cap L_2(-t^*,T;H_0^2(\Omega))$ solve (\ref{plate}) with clamped boundary conditions and appropriate initial conditions on $[0,T]$ for $i=1,2$. Then the following estimate holds for all $\e >0$, for some $\eta > 0$, and $0 \le t \le T$:
\begin{align*}
\int_0^t \big(||\Delta u||^2-||u_t||^2\big)d\tau \le&~\epsilon \int_0^t||u||^2_2 d\tau+\e + C\int_{-t^*}^0 ||u(\tau)||_2^2 d\tau + C(\epsilon,t^*,T)\sup_{\tau \in [0,t]} ||u(\tau)||^2_{2-\eta}\\&-\int_0^t<f(u),u> d\tau + |<u_t(t), u(t)>| +|<u_t(s), u(s)>|.
\nonumber
\end{align*}
Moreover, in the case where we are considering the difference $z=u^1-u^2$ of solutions solving (\ref{difference}) with $u^i(t) \in \mathscr B_R(H^2(\Omega))$ for all $t\in [0,T]$, we may utilize the estimates in Theorem~\ref{nonest} (which eliminates the stand-alone $\e$) arrive at
\begin{align}\label{zmult}
\int_s^t \big(||\Delta z||^2-||z_t||^2\big)d\tau
\le &~\epsilon \int_s^t ||z||_2^2 d\tau +C\int_{s-t^*}^t ||z(\tau)||_{2-\s}^2 d\tau + C(\epsilon,T,R)\sup_{\tau \in [0,t]} ||z(\tau)||^2_{2-\eta} \notag \\
&+ E_z(t)+E_z(s),
\end{align}
where $E_z(t)$ is given by \eqref{Ez}, i.e.,
$E_z(t) \equiv \dfrac{1}{2}\big\{||\Delta z(t)||^2 + ||z_t(t)||^2\big\}$.
\end{lemma}
The final class of estimates we  need are energy estimates for the $z$ term defines as the solution to (\ref{difference}). Energy estimates for single solutions (making use of the nonlinear potential energy) can be derived straightforwardly from (\ref{energyrelation}). The energy estimate on $z$, along with the estimate in (\ref{zmult}) above, will be used in showing asymptotic smoothness for the system.
\par
The energy relation on $[s,t]$ for $z$ in (\ref{difference}) is given by
\begin{align*}%\label{zenergyrelation}
\Ez(t)+k\int_s^t ||z_t||^2 d\tau =& ~\Ez(s) -\int_s^t <f(u^1)-f(u^2),z_t>d\tau +\int_s^t<q^z(\tau),z_t(\tau)> d\tau
\\ \notag
&
-\int_s^t<Lz(\tau),z_t(\tau)>d\tau
\end{align*}
From this, making use of splitting and Sobolev inequalities, we quickly have for $0 \le s < t \le T$, some $\eta>0$, and all $\e >0$:
\begin{align*}%\label{zenergyprelim}
\Ez(t)+k\int_s^t ||z_t||^2 d\tau \le& ~\Ez(s) +C(\epsilon,T)\sup_{\tau \in [s,t]}||z||^2_{2-\eta} +\epsilon\int_s^t\big(||z||_2^2+||z_t||^2\big)d\tau\\\nonumber &-\int_s^t <f(u^1)-f(u^2),z_t>d\tau+\Big|\int_s^t<q^z(\tau),z_t(\tau)> d\tau\Big|
\end{align*}
In the case when $k>0$ by \eqref{zmult}
this implies that
\begin{align*}%\label{zenergyprelim-2}
\frac12 \Ez(t)+c_0\int_s^t E_z d\tau \le& ~\Ez(s) +C(T,R)\sup_{\tau \in [s,t]}||z||^2_{2-\eta}
+C\int_{s-t^*}^t ||z(\tau)||_{2-\s}^2 d\tau
\\\nonumber &-\int_s^t <f(u^1)-f(u^2),z_t>d\tau+\Big|\int_s^t<q^z_t(\tau),z(\tau)> d\tau\Big|
\\ \notag
&+ |<q^z(t),z(t)>|+ |<q^z(s),z(s)>|.
\end{align*}
Here above we usee the integration by parts formula for the integral with delayed
term.
Therefore there exist $a_i>0$ and $C(T,R)>0$ such that
\begin{align*}%\label{zenergyprelim-3}
\Ez(t)+\int_s^t E_z d\tau \le& ~ a_0\left(\Ez(s)+\int_{s-t^*}^s ||z(\tau)||_{2-\s}^2 d\tau\right)
  +C(T,R)\sup_{\tau \in [s,t]}||z||^2_{2-\eta_*}
\\\nonumber &-a_1\int_s^t <f(u^1)-f(u^2),z_t>d\tau.
\end{align*}
Taking $t=T$ and integrating over $s$ in $[0,T]$ we arrive at \footnote{With other constants
$a_i$ and $C(T,R)$.}
\begin{align*}%\label{zenergyprelim-4}
T\Ez(T)+ &\int_0^Tds \int_s^T E_z d\tau \le ~ a_0\left(\Ez(0)+\int_{-t^*}^0 ||z(\tau)||_{2-\s}^2 d\tau\right)
  +C(T,R)\sup_{\tau \in [0,T]}||z||^2_{2-\eta_*}
\\\nonumber &-a_1\int_0^Tds \int_s^T <f(u^1)-f(u^2),z_t>d\tau  -a_2\int_0^T <f(u^1)-f(u^2),z_t>d\tau.
\end{align*}
Since
\[
\int_0^Tds \int_s^T E_z d\tau \ge \frac{T}2 \int_{T-t^*}^T E_z d\tau ~~~\mbox{for}~T\ge 2t^*,
\]
we arrive to the following assertion.

\begin{lemma}\label{le:observbl}
Let $u^i \in C(0,T;H_0^2(\Omega))\cap C^1(0,T;L_2(\Omega)) \cap L_2(-t^*,T;H_0^2(\Omega))$ solve (\ref{plate}) with clamped boundary conditions and appropriate initial conditions on $[0,T]$ for $i=1,2$, $T\ge 2t^*$. Additionally, assume $u^i(t) \in \mathscr B_R(H^2(\Omega))$ for all $t\in [0,T]$.
Then the following estimates
\begin{align}\label{enest1}
 \frac{T}2 \Big[\Ez(T)+ &\int_{T-t^*}^T \Ez(\tau) d\tau\Big] \le ~ a_0\left(\Ez(0)+\int_{-t^*}^0 ||z(\tau)||_{2}^2 d\tau\right)
  +C(T,R)\sup_{\tau \in [0,T]}||z||^2_{2-\eta_*}
\\\nonumber &-a_1\int_0^Tds \int_s^T <f(u^1)-f(u^2),z_t>d\tau  -a_2\int_0^T <f(u^1)-f(u^2),z_t>d\tau
\end{align}
hold with $a_i$ independent of $T$ and $R$.
\end{lemma}
\subsection{Dissipative Dynamical System}

Our next task in order to to make use of Theorem \ref{dissmooth} is to show dissipativity of the dynamical system $(S_t,\bH)$, namely that there exists a bounded, forward invariant, absorbing set. To show this, similar to the consideration in \cite[Theorem 9.3.4, p.480]{springer}, we consider the Lyapunov type function (with $\cE(u,u_t)$ as in  (\ref{fulle}) and with $\Pi_*(u)$ given by \eqref{potentiale})
\begin{align}
V(S_t y) \equiv &\cE(u(t),u_t(t))-<q^u(t),u(t)>+\nu\big(<u_t,u>+\frac{k}{2}||u||^2\big)\\\nonumber
&+\mu\Big(\int_{t-t^*}^t\Pi_*(u(s))ds+\int_0^{t^*}ds\int_{t-s}^t \Pi_*(u(\tau))d\tau\Big),
\end{align}
where $S_t y \equiv y(t)= (u(t),u_t(t),u^t)$ for $t \ge 0$ and $\mu, \nu$ are some positive numbers to be specified below.
In view of the results for the von Karman plate in \cite[Section 4.1.1]{springer}, we have that
\begin{equation}\label{energybounds}
c_0E(u,u_t) - c \le V(S_ty) \le c_1E(u,u_t)+\mu C t^*\int_{-t^*}^0 \Pi_*(u(t+\tau))d\tau +c
\end{equation}
for $\nu>0$ small enough,
 where $c_0,c_1,c,C >0$ are constant.
Here we make use  of the notation:
$
E(u,u_t) \equiv \dfrac{1}{2}||u_t||^2+\Pi_*(u).
$

 To obtain the above bound, we make direct use of our assumption on the $L_2$ bound in \eqref{assumpt1} on the term $<q^u,u>$. Additionally, we here need (and below) a critical lower bound on the potential energy, which can be found in various forms throughout \cite[p. 49 and p. 132]{springer}:
 \begin{proposition}\label{potentiallowerbound}
 For any $u \in H^2(\Omega)\cap H_0^1(\Omega)$ we have for any $\delta >0$ and any $0< \eta \le 2$
 \begin{equation*}
 ||u||^2_{2-\eta} \le \delta \big(||u||_2^2+||\Delta v(u)||^2) + C_{\delta},
 \end{equation*}
  where $v(u)=v(u,u)$ is the Airy stress function as defined in \eqref{airy-1}.
 \end{proposition}
In what follows below, we will often make use of the above theorem to give $$<[u,F_0],u>+||u||^2 \le \delta \Pi_*(u)+C_{\delta,F_0}.$$

We now compute $\ds \dfrac{d}{dt} V(S_ty)$:
\begin{align*}
\dfrac{d}{dt} V(S_ty)=&~\dfrac{d}{dt}\cE(t)+\nu<u_{tt},u>+\nu||u_t||^2+k\nu <u_t,u> \\\nonumber
&-<q^u,u_t>-<q^u_t,u>\\\nonumber
&+\mu \dfrac{d}{dt} \Big\{\int_{t-t^*}^t\Pi_*(u(s)ds+\int_0^{t^*}\int_{t-s}^t \Pi_*(u(\tau))d\tau ds \Big\}\\[.1cm]
=&~ -k\|u_{t}\|^2+<p_0-Lu,u_t> -<q^u_t,u>\\\nonumber
&~+\nu<u_{tt}+ku_t,u>+\nu ||u_t||^2 +\mu\Pi_*(u(t))-\mu\Pi(u(t-t^*))\\\nonumber
&~+\mu t^*\Pi_*(u(t))-\mu\int_{-t^*}^0\Pi_{*}(u(t-\tau))d\tau
\end{align*}
Now, we make use of the relation $$\ds u_{tt}+ku_t=-\Delta^2u+p_0+q^u(t)+[u,v(u)+F_0]-Lu$$ owing to (\ref{plate}). Substituting this back into the relation above and simplifying yields:
\begin{align*}%\label{ddtV}
\dfrac{d}{dt} V(S_ty)
=& -\big(k-\nu\big)||u_t||^2-\nu||\Delta u||^2-\nu||\Delta v(u)||^2-\mu\Pi_*(u(t-t^*))-\mu\int_{-t^*}^0\Pi_*(u(t+\tau))d\tau\\[.1cm]\nonumber
&+\nu<[u,F_0],u>+\nu<q^u,u>+<p_0-Lu,u_t>+\nu<p_0-Lu,u> \\[.1cm]\nonumber
&-<q^u_t,u>+\mu (t^*+1)\Pi_*(u(t)).
\end{align*}
In the estimate that follows we make use of (a) standard splitting via Young's inequality, (b) the bound in Proposition~\ref{potentiallowerbound}, (c) the assumption that $0<\nu< \min\{1,k\}$. Then we have for all $\epsilon>0$
\begin{align*}
\dfrac{d}{dt}V(S_ty)\le&~(-k+\nu)||u_t||^2+\Big(\dfrac{\mu (1+t^*)}{2}-\nu\Big)||\Delta u||^2+\Big(\dfrac{\mu (1+t^*)}{4}-\nu\Big)||\Delta v(u)||^2 \\\nonumber&-\mu \Pi(u(t-t^*))- \mu\int_{-t^*}^0\Pi_*(u(t+\tau))d\tau\\[.1cm]\nonumber
&+\epsilon||u_t||^2+\epsilon\big(||\Delta u||^2+||\Delta v(u)||^2\big)+C_{\epsilon,\delta,p_0,F_0}+\epsilon||q^u||^2+|<q^u_t,u>|.
\end{align*}
Now, using \eqref{assumpt1}:
$$
||q^u(t)||^2 \le~ C\int_{t-t^*}^t||\Delta u(\tau)||^2 d\tau,
$$
and \eqref{assumpt1.75} with $\psi=u$:
\begin{align*}
|<q^u_t(t),u(t)>| \le  \epsilon \left[
||\Delta u(t)||^2+ ||\Delta u(t-t^*)||^2+\int_{-t^*}^0||\Delta u(t+\tau)||^2 d\tau\right]+
C_{\epsilon,t^*}||u(t)||^2
\end{align*}
 we have (again, making use of Proposition~\ref{potentiallowerbound}) the following inequality for all $\epsilon>0$:
\begin{align*}%\label{imptineq}
\dfrac{d}{dt}V(S_ty)\le&~\big(-k+\nu+\epsilon\big)||u_t||^2+\Big(\dfrac{\mu (1+t^*)}{2}+2\epsilon-\nu\Big)||\Delta u||^2+\Big(\dfrac{\mu (1+t^*)}{4}+\epsilon-\nu\Big)||\Delta v(u)||^2 \\\nonumber
&-(\mu -2\epsilon)\Pi_*(t-t^*)-(\mu-4\epsilon)\int_{-t^*}^0\Pi_*(u(t+\tau))d\tau
+C_{\epsilon}\|u(t)\|^2+C_{\epsilon,\mu,p_0,F_0}.
\end{align*}
And, for $0<\nu<k$, and for $\mu$ and $\epsilon$ sufficiently small, we have the following lemma:
\begin{lemma}\label{le:48}
For \textit{all} $k>0$ there exist $\mu, \nu >0$ and $c(\mu,\nu,t^*,k)>0$ and $C(\mu,\nu,p_0,F_0)>0$ such that
\begin{equation*}%\label{goodneg}
\dfrac{d}{dt}V(S_ty)\le-c\big\{||u_t||^2+||\Delta u||^2+||\Delta v(u)||^2 +\Pi_*(u(t-t^*))+\int_{-t^*}^0\Pi_*(u(t+\tau))d\tau\big\}+C.
\end{equation*}
\end{lemma}

From this lemma and the upper bound in \eqref{energybounds}, we have for some $\beta>0$ sufficiently small (again, depending on $\mu$ and $\nu$): \begin{equation}\label{gronish}
\dfrac{d}{dt}V(S_ty) +\beta V(S_ty) \le C,~~t>0,
\end{equation}
The estimate above in (\ref{gronish}) implies (by a version of Gronwall's inequality) that
\begin{equation*}
V(S_ty) \le V(y)e^{-\beta t}+\dfrac{C}{\beta}(1-e^{-\beta t}).
\end{equation*}
Hence, the set
$$
\mathscr{B} \equiv \left\{y \in \bH:~V(y) \le 1+\dfrac{C}{\beta} \right\},
$$  is a bounded forward invariant absorbing set. This gives that $(\bH,S_t)$ is dissipative.

\subsection{Asymptotic Smoothness}
Recall that our dynamical system is $(S_t,\bH)$, where $S_t$ is the evolution operator corresponding to plate solutions to (\ref{plate}) and $\bH = H_0^2(\Omega) \times L_2(\Omega) \times L_2(-t^*,0;H_0^2(\Omega))$. To show asymptotic smoothness of this dynamical system, we will make use of an abstract Theorem~\ref{psi}.

To make use of this theorem, we will consider our functional $\Psi$ to be comprised of lower order terms (compact with respect to $\bH$) and quasicompact ($\int_s^t <f(u^1)-f(u^2),z_t>d\tau$) terms. On the LHS of the above estimate, we see that we need to produce an estimate which bounds trajectories in $\bH$, i.e. $||(u(t),u_t(t),u^t)||^2_{\bH}$ (taking the metric $d$ to be $||\cdot||_{\bH}$). Such an estimate will be produced below by combining our energy estimates produced earlier:
\begin{lemma}\label{le:khan}
Suppose $z=u^1-u^2$ is as in (\ref{difference}), with $y^i(t)=(u^i(t),u_t(t)^i,u^{t,i})$ and $y^i(t) \in \mathscr B_R(\bH)$ for all $t\ge 0$. Also, let $\eta >0$ and $\Ez(t)$ be defined as  in (\ref{Ez}).
Then for every $0<\e<1$ there exists  $T=T_\e(R)$ such that the following estimate
\begin{equation*}%\label{smoothest}
E_z(T) + \int_{T-t^*}^T ||z(\tau)||_2^2 d\tau \le \epsilon + \Psi_{\epsilon,T,R}(y^1,y^2)
\end{equation*}
 holds,
where \begin{align*} \Psi_{\epsilon,T,R}(y^1,y^2) \equiv& C(R,T) \sup_{\tau \in [0,T]} ||z(\tau)||_{2-\eta}^2 +a_1\Big| \int_0^T<f(u^1(\tau))-f(u^2(\tau)),z_t(\tau)>d\tau\Big| \\\nonumber& + a_2\Big|\int_0^T \int_s^T <f(u^1(\tau))-f(u^2(\tau)),z_t(\tau)>d\tau ds\Big|.\end{align*}
\end{lemma}
\begin{proof} It follows from \eqref{enest1} by dividing by $T$ and
taking $T$ large enough.
\end{proof}
In Lemma \ref{le:khan} above, we have obtained the necessary estimate for asymptotic smoothness; it now suffices to show that $\Psi$, as defined above, has the compensated compactness condition described in Theorem \ref{psi}.

Before proceeding, let us introduce some notation which will be used throughout the remainder of this section and in the following section. We will write \begin{equation}\label{notations} l.o.t. = \sup_{\tau \in [0,T]}||z(\tau)||^2_{2-\eta},~~
\mathcal{F}(z)= f(u^1)-f(u^2). \end{equation}

\begin{theorem}\label{smoothness}
The dynamical system $(S_t,\bH)$ generated by weak solutions to (\ref{plate}) is asymptotically smooth.
\end{theorem}
\begin{proof}
In line with the discussion above, we aim to make use of Theorem \ref{psi}. To do so, it suffices to show the compensated compactness condition for $\Psi_{\epsilon,T,R}$ which we now write as $\Psi$, with $\epsilon, T,$ and $R$ fixed along with the other constants given by the equation. Let $B$ be a bounded, positively invariant set in $\bH$, and let $\{y_n\}\subset B \subset \mathscr{B}_R(\bH)$. We would like to show that
$$
\liminf_m \liminf_n \Psi(y_n,y_m) = 0.
$$
More specifically, for any initial data $U_0^1=(u^1_0,u^1_1,\eta^1),
~U^2_0=(u^2_0,u^2_1,\eta^2) \in B $ (where $\eta^i$ belongs to $L_2(-t^*,0: H^2_0(\Om))$) we define
\begin{equation*}
\widetilde{ \Psi}^2 (U^1_0, U^2_0) = \Big|\int_0^T <\mathcal{F}(z)(\tau),
z_t (\tau)>d\tau\Big| +\Big |\int_0^T \int_s^T <\mathcal{F}(z(\tau)) ,z_t(\tau) >d\tau ds \Big|
\end{equation*}
where the function $z = u^1 -u^2 $ has initial data $U^1_0-U^2_0$ and solves (\ref{difference}). The key to
compensated compactness is the following representation for the bracket \cite[pp. 598-599]{springer}:
\begin{align*}
<\mathcal{F}(z)(\tau) ,z_t(\tau)>=& ~\frac{1}{4} \frac{d}{d\tau} \big\{ - ||\Delta v(u^1) ||^2
- ||\Delta v(u^2) ||^2 + 2 <[z,z],F_0> \big\}\\\nonumber&-< [ v(u^2),u^2], u^1_t> - %
< [ v(u^1) , u^1 ], u^2_t>.
\end{align*}
Integrating the above expression in time and evaluating on the difference of
two solutions $z^{n,m} =w^n-w^m $ with initial data $W_0^n-W_0^m$, where $w^{i} \rightharpoonup w$, yields:
\begin{align}\label{itlim}
\lim_{n\to \infty}\lim_{m\to \infty} \int_s^T <\mathcal{F}(z^{n,m})(\tau)
,z_t^{n,m}(\tau) > d\tau= \dfrac{1}{2} \big\{ ||\Delta v(w(s)) ||^2 - ||\Delta v(w(T))||^2 \big\} &  \\\nonumber
- \lim_{n \to \infty} \lim_{m \to \infty} \int_s^T \big\{ < [ v(w^n),
w^n ], w_t^m> + < [ v(w^m), w^m ], w_t^n>\big\} &,
\end{align}
where we have used (i) the weak convergence in $H^2(\Omega)$ of $z^{n,m} $
to 0, and (ii) compactness of $\Delta v(\cdot) $ from $H^2(\Omega) \rightarrow
L_2(\Omega) $ as in Lemma~\ref{l:airy-1}. The iterated limit in (\ref{itlim}) is handled via iterated weak
convergence, as follows:
\begin{equation*}
\lim_{n \to \infty} \lim_{m \to \infty} \int_s^T \big\{ < [ v(w^n), w^n
], w_t^m> + < [ v(w^m), w^m ], w_t^n>\big\}
\end{equation*}
\begin{equation*}
= 2 \int_s^T <[ v(w), w] , w_t> = \frac{1}{2} ||\Delta v(w)(s) ||^2 - \frac{%
1}{2} || \Delta v(w)(T) ||^2.
\end{equation*}
 This yields the desired conclusion, that
\begin{equation*}
\lim_{n\rightarrow \infty }\lim_{m\rightarrow \infty }\int_{s}^{T}(\mathcal{F%
}<z^{n,m}(\tau)),z_{t}^{n,m}(\tau)>d\tau=0.
\end{equation*}%
The second integral term in $\widetilde{ \Psi }^2$ is handled similarly.
Since the term $l.o.t.$ above is
compact (below energy level) via the Sobolev embedding,
as a
consequence we obtain
\begin{equation*}
\underset{m\rightarrow \infty }{\lim \inf }\underset{n\rightarrow \infty }{%
\lim \inf }\widetilde{ \Psi }(y_{n},y_{m})=0.
\end{equation*}
 This concludes the proof of the asymptotic smoothness via Theorem \ref{psi}.
\end{proof}
Having shown the asymptotic smoothness property, we can now conclude by Theorem \ref{dissmooth} that there exists a compact global attractor $\bA \subset \bH$ for the dynamical system $(S_t, \bH)$.

\subsection{Quasistability Estimate}
In this section we refine our methods in the asymptotic smoothness calculation and work on trajectories from the attractor, whose existence has been established in the previous sections.

Analyzing \eqref{enest1}, we may also write
\begin{align}\label{eq-obsr1}
T\left[E_z(T)+\int_{T-t^*}^TE_z(\tau)d\tau\right] \le c\big(E_z(0)+\int_{-t^*}^0||z(\tau)||_2^2d\tau\big)
\\
+
C\cdot T \sup_{s\in [0,T]}\Big|\int_s^T<\cF(z),z_t>d\tau \Big|+C(R,T)\sup_{\tau \in [s,t]} ||z||^2_{2-\eta},
\notag
\end{align}
where $\cF(z)$ is given in \eqref{notations}.
We note that $c$ does not depend on $T\ge \min\{1,2t^*\}$, and $l.o.t.$ is of quadratic
order.

In order to prove the quasistability estimate (as in \eqref{quasi}), we have to handle the non-compact term $<
\mathcal{F}(z),z_{t}> $. We recall the  relation \eqref{eq4.5}  in Theorem \ref{nonest}:
if  $u^i \in C(s,t;H^2(\Omega))\cap C^1(s,t;L_2(\Omega))$ with $u^i(\tau) \in \mathscr{B}_R(L_2(\Omega))$
for $\tau\in [s,t]$, then
\begin{align}\label{1new}
\Big|
\int_s^t <\cF(z),z_t(\tau)>d\tau \le &~C(R)\sup_{\tau \in [s,t]} ||z||^2_{2-\eta}+C\frac12\Big|
 \int_s^t P(z(\tau))d\tau\Big|
 \end{align} for some $0<\eta<1/2$.
Here $P(z)$ is given by (\ref{4.9aa}).
\par
Let $\gamma_{u^1} =\{(u^1(t),u^1_{t}(t),[u^1]^t):t\in \mathbb{R\}}$ and $\gamma_{u^2} =\{(u^2(t),u^2_{t}(t),[u^2]^t):t\in \mathbb{R\}}$ be trajectories from the
attractor $\mathbf{A}$. It is clear that for the pair $%
u^1(t)$ and $u^2(t)$ satisfy the hypotheses of the estimate in \eqref{1new} for every
interval $[s,t]$.
Our main goal is to handle the second term on the right hand side of (\ref{1new}) which is of {\it critical
regularity}.
 To accomplish this we shall use the already established {\it compactness} of the attractor in the state space $\bH=H_0^2(\Omega)\times L_2(\Omega) \times L_2(-t^*,0;H_0^2(\Omega))$.
\par
Since for every $\tau \in \mathbb{R}$, the element $u^i_{t}(\tau )$ belongs to
a compact set in $L_{2}(\Omega )$, by density of $H_0^2(\Omega) $ in $L_2(\Omega) $
we can assume, without a loss of generality, that  for every $\epsilon >0$ there exists a finite set $%
\{\phi _{j}\}\subset H_{0}^{2}(\Omega )$ , $ j = 1,2,...,n(\epsilon) $, such that  for all $\tau \in \mathbb{R} $ we can find indices $%
j_{1}(\tau) $ and $j_{2} (\tau) $ so that
\begin{equation*}%\label{2new}
||u^1_{t}(\tau )-\phi _{j_{1}(\tau)}||+||u^2_{t}(\tau)-\phi _{j_{2}(\tau)}||\leq \epsilon
~~\mbox{ for all } ~\tau \in \mathbb{R}.
\end{equation*}
Let $P(z)$ be given by (\ref{4.9aa}) with the pair $u^1(t)$ and $u^2(t)$ and
\begin{equation*}
P_{j_{1},j_{2}}(z)\equiv -\left( \phi _{j_{1}},[u^1,v(z)]\right) -\left( \phi
_{j_{2}},[u^2,v(z,z)]\right) -\left( \phi _{j_{1}}+\phi
_{j_{2}},[z,v(u^1+u^2,z)]\right),
\end{equation*}%
where $z(t)=u^1(t)-u^2(t).$ It can be easily shown that for all $j_1, j_2 \leq n(\epsilon) $
\begin{equation}\label{3new}
||P(z(\tau))-P_{j_{1}(\tau),j_{2}(\tau)}(z(\tau))||\leq \epsilon C(\bA)||z(\tau )||_{2}^{2}
\end{equation}
uniformly in $\tau \in  \mathbb{R} $.
\par
Starting with  the estimate  (1.4.17) page 41 \cite{springer},
\begin{equation*}%\label{kar1}
||[u,w]||_{-2} \leq C ||u||_{2-\beta}||w||_{1+\beta} ,~~  \forall \beta \in [0,1)
\end{equation*}
and exploiting elliptic regularity
one obtains
\begin{equation}\label{kar2}
||[u,v(z,w)]||_{-2}  \leq C ||u||_{2-\beta} || [z,w]||_{-2}  \leq C ||u||_{2-\beta} ||z||_{2-\beta_1} ||w||_{1+\beta_1 },
 \end{equation}
 where above inequality holds for any  $\beta, \beta_1 \in [0,1) $
\par
 Recalling the  additional smoothness of $\phi_j \in H_0^2(\Omega)$, along with the  estimate in (\ref{kar2})
 applied with $\beta = \beta_1 =\eta $,
 and accounting  the structure of $P_j$ terms one obtains:
\begin{equation*}
||P_{j_{1},j_{2}}(z)||\leq C(\bA)\big( ||\phi _{j_{1}}||_{2}+||\phi
_{j_{2}}||_{2}\big) ||z(\tau )||_{2-\eta }^{2}
\end{equation*}%
for some $0<\eta<1$. So we have
\begin{equation}\label{5.5}
\underset{j_{1},j_{2}}{\sup }||P_{j_{1},j_{2}}(z)||\leq
C(\epsilon)||z(\tau )||_{2-\eta }^{2}~~~\mbox{for some $0<\eta<1$},
\end{equation}%
where $C(\epsilon) \rightarrow \infty $ when $\epsilon \rightarrow 0 $.
Taking into account (\ref{3new}) and (\ref{5.5}) in (\ref{1new})  we obtain
\begin{equation}\label{4.12}
%\underset{t\in \lbrack 0,T]}{\sup }
\Big|\int_{s}^{t}< \mathcal{F}%
(z),z_{t}> \Big|\leq C(\epsilon ,T,\bA)\underset{\tau \in
\lbrack s,t]}{\sup }||z(\tau)||_{2-\eta }^{2}+\epsilon
\int_{s}^{t}||z(\tau )||_{2}^{2}d\tau
\end{equation}%
for all $s\in \mathbb{R}$ with $\eta >0$ and $t>s$.
Considering \eqref{4.12} and taking $T$ sufficiently large, we have
from \eqref{eq-obsr1}
\begin{equation*}
E_{z}(T)+\int_{T-t^*}^{T}||z(\tau)||_2^2 d\tau\leq \gamma \big(E_z(0)+\int_{-t^*}^0||z(\tau)||_2^2d\tau)+C(\mathbf{A},T,k,t^*)\underset{\tau \in \lbrack0,T]}{\sup }||z(\tau )||_{2-\eta }^2
\end{equation*}%
with $\gamma<1$.
Thus by the standard argument (see \cite{springer}) we finally conclude that for $y(t)=(z(t),z_t(t),z^t)$
\begin{equation*}
||y(t)||_{\bH}^2 \le C(\sigma,\bA)||y(0)||_{\bH}^2e^{-\sigma t}+C \sup_{\tau \in [0,t]} ||z(\tau)||_{2-\eta}^2.
\end{equation*}
Hence
on the strength of Theorem \ref{t:FD}, applied
with $B =\mathbf{A}$
and $$
\bH = H_0^2(\Omega) \times
L_2(\Omega) \times L_2(-t^*,0;H_0^2(\Omega)),
$$
where $H_1 = H^{2-\eta}(\Omega)$, we
conclude that $\mathbf{A}$ has a finite fractal dimension.
\par
Additionally, Theorem \ref{t:FD} guarantees that $$||u_{tt}(t)||^2+||u_t(t)||_2^2 \le C~\text{ for all } t \in \R.$$
Since $u_{t}\in H^{2}(\Omega )\subset C(\Omega)$, elliptic regularity theory for
$$
\Delta ^{2}u=-u_{tt}-k u_t-f(u)-Lu+q(u^t,t)
$$
with the clamped boundary conditions give that $$\ ||u(t)||_{4}^{2}\leq C ~\text{for all} ~
t\in \reals.$$
Thus, we can conclude additional regularity of the trajectories from the attractor $\bA \subset \bH$
stated in Theorem~\ref{maintheorem}.
\par
We have now completed the proof of Theorem \ref{maintheorem}.

\subsection{Extensions and Open Problems}
In this section, we briefly mention possible extensions of our results and open problems pertaining to the models discussed above.
\begin{itemize}
\item Other homogeneous boundary conditions: hinged, simply supported, free-type, or combinations thereof.
\item Nonlinear internal damping, i.e. considering $k(u_t)$ in the plate equation, where $k(\cdot)$ is a Nemitsky operator.
\item Boundary damping via (nonlinear) hinged dissipation (\cite{ACC}).
\item Other physical nonlinearities; replacing the von Karman nonlinearity in the considerations above with Berger or Kirchoff-type nonlinearity (as discussed in \cite{oldchueshov1,ch-0,oldchueshov2}). This should not present major technical issues.
\item Convergence to equilibria; one may conjecture that individual trajectories of the full flow-plate system converge to single elements of the attractor. However, Dowell's conjecture (supported by numerics) states that
only in the subsonic case  solutions stabilize to stationary points, and
in the supersonic case,  locally stable periodic (or even chaotic) orbits are possible.
Hence, it is likely that one can discuss this convergence only in the subsonic case.
The principal issue here is finiteness of the dissipation integral
\[
\int_0^\infty ||u_t||^2 dt
\]
We can easily guarantee this if we have additional
plate damping in the coupled system. In this case we can achieve stabilization
in the same way as it done in \cite{springer} for the rotational case (see also \cite{chuey}).
So the issue becomes how to obtain some form of finiteness
of dissipation integral in the case where full energy of coupled system is preserved (the reduction result Theorem \ref{rewrite} is not employed). This issue
remains open (see also the corresponding remark in \cite[Section 12.4.2]{springer}).
\end{itemize}

\section{Flow-Plate Interactions}\label{flowplateint}
In this section we provide sketch of the proof of Theorem~\ref{rewrite},
which is needed for our principal application of Theorem ~\ref{maintheorem}
to the long-time dynamics of  the fully coupled model:
\begin{equation}\label{flowplate-1}\begin{cases}
u_{tt}+\Delta^2u+ku_t+f(u)= p_0+\big(\partial_t+U\partial_x\big)\gamma[\phi] & \text { in } \Omega\times (0,T),\\
u(0)=u_0;~~u_t(0)=u_1,\\
u=\Dn u = 0 & \text{ on } \partial\Omega\times (0,T),\\
(\partial_t+U\partial_x)^2\phi=\Delta \phi & \text { in } \realsthree_+ \times (0,T),\\
\phi(0)=\phi_0;~~\phi_t(0)=\phi_1,\\
\Dn \phi = -\big[(\partial_t+U\partial_x)u (\xb)\big]\cdot \mathbf{1}_{\Omega}(\xb) & \text{ on } \realstwo_{\{(x,y)\}} \times (0,T).
\end{cases}
\end{equation}
\begin{proof}
By Theorem \ref{well-U}, the system (\ref{flowplate-1})
generates a strongly continuous  semigroup $T_t$ on $H$. This is to say
that $ (\phi(t), \phi_t(t), u(t), u_t(t) )
= T_t(\phi_0,\phi_1,u_0, u_1),  t > 0 $.
The main idea behind the proof of Theorem~\ref{rewrite} (see \cite[Section 6.6.5]{springer}) is to split gas flow variable $\phi$ in two components:
$\phi(\xb,t) = \phi^*(\xb,t)+\phi^{**}(\xb,t)$,
where $\phi^*(\xb,t)$ solves \eqref{flow} with $d(\xb,t)\equiv 0$
and  $\phi^{**}(\xb,t)$ is solution to non-homogenous problem
\eqref{flow} with the zero initial data $\phi_0=0$ and $\phi_1=0$.
Here we note that  with  $(\phi_0,\phi_1) \in H^1(\R_+^3)\times  L_2(\R_+^3) $
one obtains \cite{miyatake1973,supersonic}
 $(\phi^*(t),  \phi^{*}_t(t)) \in   H^1(\R_+^3)\times  L_2(\R_+^3)$. Thus, by Theorem \ref{well-U}
 we also have that
 $$
 (\phi^{**}(t),  \phi^{**}_t(t) \in  H^1(\R_+^3)\times  L_2(\R_+^3).
 $$
 Note that \textit{this last property is not valid for a flow solution with $L_2$ boundary Neumann data}\footnote{
 The general theory will provide at most $ H^{2/3}(\R_+^3\times [0,T])$.}.  However, the improved regularity is due to the interaction with the plate and the resulting cancelations on  the interface.
  Moreover, we also obtain a meaningful ``hidden trace regularity" for the aeroelastic potential on the boundary of the structure \cite{supersonic}:
  \begin{equation}\label{trace}
  (  \partial _{t} + U \partial_{x} )\gamma [ \phi ] \in L_2(0, T; H^{-1/2}(\Omega) )
  \end{equation}
  where $T$ is arbitrary.

Then, using the Kirchhoff type representation for the solution  $\phi^*(\xb,t)$
in $\R_+^3$ (see, e.g., Theorem~6.6.12 in \cite{springer}), we can conclude that
if the initial data   $\phi_0$ and $\phi_1$ are  localized in the ball $K_R=\{\xb\in \R_+^3: |\xb|\le R\}$,
then by  finite dependence on the domain of the signal in three dimensions  (Hyugen's principle),
   one obtains  $\phi^*(\xb,t)\equiv 0$ for all $\xb\in K_R$
and $t\ge t_R$. Thus we have that
\[
\big(\partial_t+U\partial_x\big)\gamma[\phi^*]\equiv0,~~~\xb\in \Om,~t\ge t_R.
\]
Thus it remains  to consider flow variable $\phi^{**}$, whose aeroelastic potential on the boundary coincides with
that of $\phi$, and hence it displays regularity as in (\ref{trace}) . This allows one to perform calculations
as in  \cite[Theorem~6.6.10]{springer})  in order
to obtain the representation
\begin{align*}%\label{ch6.2.3.23}
 ( \partial _{t} + U \partial_{x} )\gamma [ \phi ] = & -\, d(x,y,t)
\\ &
+\, \frac{1}{2\pi}
 \int_{0} ^t  ds \int_{0}^{2\pi} d\theta
[M_\theta d] (x - (U+\sin\theta)s,
y - s\cos\theta, t-s). \nonumber
\end{align*}
Now using the same calculations as in \cite[p.333]{springer}
we arrive at  the following equation:
\begin{equation}\label{reducedplate-1}
u_{tt}+\Delta^2u+ku_t-[u,v(u)+F_0]=p_0-(\partial_t+U\partial_x)u-q^u(t)
\end{equation}
for $t$ large enough, with
\begin{equation}\label{potential-1}
q^u(t)=\dfrac{1}{2\pi}\int_0^{t^*}ds\int_0^{2\pi}d\theta [M^2_{\theta}\widehat u](x-(U+\sin \theta)s,y-s\cos \theta, t-s).
\end{equation}
Here, $\widehat u$ is the extension of $u$ by 0 outside of $\Omega$; $M_{\theta} = \sin\theta\partial_x+\cos \theta \partial_y$ and $t^*$ is given by
\eqref{delay}.
\end{proof}

\section{Appendix: Properties of delayed force $q$ }\label{appendix}
In this Appendix we consider properties of the delayed (aerodynamic type) force
and prove Proposition~\ref{pr:q} and Lemma~\ref{le:q}.
\subsection{Hidden Compactness of Retarded Potential: Proof of Proposition~\ref{pr:q}}
The proof of the bounds \eqref{qnegest}--\eqref{qnegest3} can be found in \cite{Chu92b} and \cite{springer}.
Thus we need to check \eqref{qnegest4} only. Without loss of generality we can assume $u\in C(-t^*,+\infty;
C_0^\infty(\Om))$.
\par
For any $\psi \in H^1_0(\Omega)$ we have
\begin{align}
<q_t^u(t),\psi> =& \Big<\int_0^{2\pi}\frac{1}{2 \pi}[M^2_{\theta}\widehat u]\big(\xb(U,\theta,0),t \big)d\theta,\psi\Big>\\\nonumber
&-\Big<\int_0^{2\pi}\frac{1}{2 \pi}[M^2_{\theta}\widehat u]\big(\xb(U,\theta,t^*),t-t^* \big)d\theta,\psi \Big>\\\nonumber
&+\Big<\Big( \int_0^{t^*}\int_0^{2\pi}(U+\sin\theta)\frac{1}{2 \pi}[M^2_{\theta}\widehat u]_x\big(\xb(U,\theta,s),t-s \big)d\theta ds\Big),\psi\Big>\\\nonumber
&+\Big< \Big(\int_0^{t^*}\int_0^{2\pi}(\cos\theta)\frac{1}{2 \pi}[M^2_{\theta}\widehat u]_y\big(\xb(U,\theta,s),t-s \big)d\theta ds\Big),\psi\Big>,
\end{align}
recalling that $\xb(U,\theta,s) = (x-(U+\sin\theta)s,y-s\cos\theta)$.
 In all integrals above we extend the integration over $\Omega$ to all of $\realstwo$ and change spatial variables.
\begin{align*}
\Big|<q^u_t(t),\psi> \Big|\le &~ C\Big\{\Big|\int_0^{2\pi}\int_{\realstwo}[M_\theta^2 \widehat u](\tau)\psi d\xb ~d\theta \Big| \\\nonumber
&+\Big|\int_0^{2\pi}\int_{\realstwo}[M_\theta^2 \widehat u](t-t^*)\psi(\xb(U,\theta, -t^*)) d\xb ~d\theta \Big| \\\nonumber
&+\Big| \int_0^{t^*}\int_0^{2\pi}\int_{\realstwo}(U+\sin \theta)[M_\theta^2 \widehat u]_x(\xb,\tau-s)\psi(\xb(U,\theta, -s)) d\xb ~d\theta~ds \Big| \\\nonumber
&+\Big| \int_0^{t^*}\int_0^{2\pi}\int_{\realstwo}\cos \theta[M_\theta^2 \widehat u]_y(\xb,\tau-s)\psi(\xb(U,\theta, -s)) d\xb ~d\theta~ds \Big|\Big\}.
\end{align*}
Now, we note that $M_{\theta}$ can be moved under the integration in $\xb$, and we have $|M_{\theta}f(\xb)| \le |f_x(\xb)|+|f_y(\xb)|$ for all $f$. Hence, we integrate by parts with a single $M_{\theta}$ in both of the first integrals; in the third and fourth integrals, we also integrate by parts once as well. This leaves us with:
$$
|<q^u_t(t),\psi>| \le C\Big\{ ||u(t)||_1+||u(t-t^*)||_1+\int_{-t^*}^0||u(t+\tau)||_2d\tau\Big\}||\psi||_1.
$$
This implies the conclusion in \eqref{qnegest4}. The proof of Proposition~\ref{pr:q} is complete.

\subsection{Proof of Lemma~\ref{le:q} }
 The relation in \eqref{hidden1} easily follows from
 \eqref{assumpt1-0} and simple formula:
 \[
 \int_0^t d\tau \int_{\tau-t*}^\tau \phi(s) ds\le t^* \int_{-t*}^t \phi(s) ds,~~~\forall~\phi\in L_1(0,T).
 \]
 Now we prove \eqref{hidden2}.
  Integrating by parts in $t$, and then applying \eqref{assumpt1.75} with $\psi = u(t)$, we have:
\begin{align*}%\label{needed}
\Big|\int_0^t<q^u(\tau),u_t(\tau)>d\tau \Big|\le & C\Big\{\int_0^t||u(\tau)||_2 ||u(\tau)||_{2-\eta}d\tau +\int_0^t ||u(\tau-t^*)||_2||u(\tau)||_{2-\eta} d\tau \\ \nonumber
&+\int_0^t \int_0^{t^*}||u(\tau-s)||_2||u(\tau)||_{2-\eta} dsd\tau\\ \nonumber
&+||q^u(t)||_{-\s}||u(t)||_\s+||q^u(0)||_{-\s}||u(0)||_\s\Big\}\\ \nonumber
\le
 & ~\epsilon \int_{-t^*}^t||u(\tau)||_2^2+ C_{\epsilon} T\sup_{0,t}||u(\tau)||_{2-\eta}^2
\\\nonumber
&+\epsilon\int_0^t\int_0^{t^*}||u(\tau-s)||_2^2 ds~ d\tau
 +C_{\epsilon}t^*\int_0^t||u(\tau)||_{2-\eta}^2d\tau \\[.1cm]\nonumber
&+C\big\{\e ||q^u(t)||_{-\s}^2 + \frac1{\e}||u(t)||_\s^2+\e||q^u(0)||_{-\s}^2 +\frac1{\e}||u(0)||^2_\s\big\}.
\end{align*}
After rescaling of $\e$ this implies
\begin{align*}
\Big|\int_0^t<q^u(\tau),u_t(\tau)>d\tau \Big|
\le &~C(t^*,\epsilon)(1+ T) \sup_{\tau \in [0,t]} ||u(\tau)||_{2-\eta}^2+\epsilon\int_{-t^*}^t  ||u(\rho)||_2^2 d\rho d\tau\\
&+\e\big\{ ||q^u(t)||_{-\s}^2+||q^u(0)||_{-\s}^2\big\}.
\end{align*}
Therefore by \eqref{assumpt1}
this implies \eqref{hidden2} with $\eta_*=\min\{\eta, 2-\s\}$.

\end{document}